\documentclass[12pt,leqno]{article}
\usepackage{amsfonts}
\usepackage{amsmath, amsthm, amsfonts, amssymb, color}
\usepackage{mathrsfs}
\usepackage{color}
\newtheorem{thm}{Theorem}[section]

\newtheorem{lem}[thm]{Lemma}

\newtheorem{exa}[thm]{Example}

\theoremstyle{definition}

\newcommand{\scr}[1]{\mathscr #1}
\definecolor{wco}{rgb}{0.5,0.2,0.3}

\numberwithin{equation}{section} \theoremstyle{remark}

\newcommand{\ua}{\uparrow}

\title{{\bf Isoperimetric Inequalities for Non-Local Dirichlet Forms}
}
\author{
{\bf   Feng-Yu Wang$^{a), b)}$  and Jian Wang$^{c)}$}\\
\footnotesize{$^{a)}$Center of Applied Mathematics, Tianjin University, Tianjin 300072, China}\\
 \footnotesize{$^{b)}$Department of Mathematics,
Swansea University, Singleton Park, SA2 8PP, United Kingdom}\\
\footnotesize{$^{c)}$ College of Mathematics and Informatics, 
Fujian Normal University, Fuzhou 350007, China}\\
\footnotesize{wangfy@tju.edu.cn,
F.-Y.Wang@swansea.ac.uk,
 jianwang@fjnu.edu.cn}}
\begin{document}
\allowdisplaybreaks
\def\R{\mathbb R}  \def\ff{\frac} \def\ss{\sqrt} \def\B{\mathbf
B}
\def\N{\mathbb N} \def\kk{\kappa} \def\m{{\bf m}}
\def\ee{\varepsilon}
\def\dd{\delta} \def\DD{\Delta} \def\vv{\varepsilon} \def\rr{\rho}
\def\<{\langle} \def\>{\rangle} \def\GG{\Gamma} \def\gg{\gamma}
  \def\nn{\nabla} \def\pp{\partial} \def\E{\scr E}
\def\d{\text{\rm{d}}} \def\bb{\beta} \def\aa{\alpha} \def\D{\scr D}
  \def\si{\sigma} \def\ess{\text{\rm{ess}}}
\def\beg{\begin} \def\beq{\begin{equation}}  \def\F{\scr F}
\def\Ric{\text{\rm{Ric}}} \def\Hess{\text{\rm{Hess}}}
\def\e{\text{\rm{e}}} \def\ua{\underline a} \def\OO{\Omega}  \def\oo{\omega}
 \def\tt{\tilde} \def\Ric{\text{\rm{Ric}}}
\def\cut{\text{\rm{cut}}} \def\P{\mathbb P} \def\ifn{I_n(f^{\bigotimes n})}
\def\C{\scr C}      \def\aaa{\mathbf{r}}     \def\r{r}
\def\gap{\text{\rm{gap}}} \def\prr{\pi_{{\bf m},\varrho}}  \def\r{\mathbf r}
\def\Z{\mathbb Z} \def\vrr{\varrho} \def\ll{\lambda}
\def\L{\scr L}\def\Tt{\tt} \def\TT{\tt}\def\II{\mathbb I}
\def\i{{\rm in}}\def\Sect{{\rm Sect}}  \def\H{\mathbb H}
\def\M{\scr M}\def\Q{\mathbb Q} \def\texto{\text{o}} \def\LL{\Lambda}
\def\Rank{{\rm Rank}} \def\B{\scr B} \def\i{{\rm i}} \def\HR{\hat{\R}^d}
\def\to{\rightarrow}\def\l{\ell}
\def\EE{\scr E}\def\1{{\bf 1}}
\def\A{\scr A}\def\Subset{\subset} \def\Hup{\sup}\def\Ss{\ss}\def\Hetminus{\setminus}
\def\BB{\scr B}\def\H{\scr H}\def\sN{\scr N}

\maketitle

\begin{abstract} Let $(E,\F,\mu)$ be a $\si$-finite measure space. For
a non-negative symmetric measure $J(\d x, \d y):=J(x,y) \,\mu(\d x)\,\mu(\d y)$ on $E\times E,$ consider the quadratic form
$$\E(f,f):= \frac{1}{2}\int_{E\times E}  (f(x)-f(y))^2 \, J(\d x,\d y)$$ in $L^2(\mu)$. We characterize  the relationship between  the isoperimetric  inequality and  the super Poincar\'e inequality  associated with $\E$.  In particular, sharp Orlicz-Sobolev type and Poincar\'e type isoperimetric inequalities are derived for  stable-like Dirichlet forms on $\R^n$, which include the existing fractional isoperimetric inequality as a special example.
     \end{abstract} \noindent
 \vskip 2cm

\section{Introduction}

For local (i.e.\ differential) quadratic forms, the isoperimetric inequality is a geometric inequality using the surface area of a set to bound its volume, see, for instance  \cite{Con, Ma, Mil} and references therein, for the study of isoperimetric inequalities and  applications to symmetric diffusion processes.  In this case,   the surface area   refers to the possibility for the associated diffusion process to exit the set.

In the non-local case, the associated process is a jump process which exits a set without hitting the boundary,   so it is reasonable to   replace  the surface area of a set $A$ by the jump rate from $A$ to its complementary $A^c$. In this spirit,   the famous Cheeger inequality \cite{Ch70} for  the first eigenvalue was extended in \cite{LS, CW00} to jump processes (see also \cite{Sal} for finite Markov chains).
See \cite{C1, C2, W00a, W00c,  RW, WZ, M08, M09} for the study of more general functional inequalities of symmetric jump processes using isoperimetric constants. These references  only consider large jumps   (i.e. the total jump rate is finite).    In this paper, we aim to investigate isoperimetric inequalities  for non-local forms with infinite jump rates, for which small jumps will paly a key role.

 \

To explain our motivation more clearly, let us start from the following classical isoperimetric inequality   on $\R^n$:
\beq\label{LC}\mu_\pp(\pp A)\ge n \mu(A)^{\ff{n-1}n} \oo_n^{\ff 1 n},\end{equation}
where $A$ is a measurable subset of $\R^n$ with finite volume, $\pp A$ is its boundary, $\oo_n$ is the volume of the $n$-dimensional unit ball,   $\mu$ is the Lebesgue measure and $\mu_\pp$ is the   area measure induced by $\mu$:
$$
\mu_\pp(\pp A):= \limsup_{\vv\downarrow 0} \ff{\mu(\{{\rm dist}(\cdot, A)\le \vv\})-\mu(A)}{\vv}. $$
In particular,  the equality in \eqref{LC} holds for $A$ being a ball.   By the co-area formula, \eqref{LC} is equivalent to the sharp $L^1$-Sobolev inequality (i.e.\ the energy form is of $L^1$ type)
\beq\label{S9} \|f\|_{\ff{n}{n-1}}\le    \ff 1 {n\oo_n^{1/ n}} \int_{\R^n} |\nn f (x)| \,\d x,\quad f\in W^{1,1}(\R^n),  \end{equation} where for any $p\ge 1$,  $\|f\|_p:=\big(\int_{\R^n} |f(x)|^{p}\,\d x\big)^{1 /p}$ and  $W^{1,p}(\R^n)$ is the homogeneous Sobolev space of differentiability $1$ and
integrability $p$.
For $n>2$,
applying \eqref{S9} to $f=|g|^{\ff{2(n-1)}{n-2}}$  and using the Cauchy-Schwarz inequality, we obtain the sharp Sobolev inequality:
$$\|g\|_{\ff {2n}{n-2}}  \le \ff {2(n-1)} {n(n-2)\oo_n^{ 1/ n}} \bigg(\int_{\R^n} |\nn g (x)|^2\, \d x\bigg)^{1/ 2},\quad g\in W^{1,2}(\R^n).$$

These inequalities are also available for the $\aa$-stable Dirichlet form. For any $\aa\in (0,2\land n)$, there exists a universal constant $C>0$ such that the fractional Sobolev inequality
$$\|f\|_{\ff{2n}{n-\aa}}  \le C \left(\int_{\R^n\times \R^n} \ff{(f(x)-f(y))^2}{|x-y|^{n+\aa}}\,\d x\,\d y\right)^{{1}/{2}},\quad f\in C_c^\infty(\R^n)   $$ holds.  By an approximation argument, this inequality can be extended to $f\in  W^{\aa/2,2}(\R^d).$  Here and in what follows, for $p\ge 1$, $W^{\aa/2,p}(\R^d)$ is denoted by the fractional homogeneous Sobolev space, which is the completion of $C_c^\infty(\R^n)$ with respect to $$\left(\int_{\R^n\times \R^n} \ff{|f(x)-f(y)|^p}{|x-y|^{n+p\aa /2}}\,\d x\,\d y\right)^{{1}/{p}}.$$
Correspondingly to \eqref{S9}, \cite[Theorem 1.1]{MS} (or \cite[Theorem 4.1]{FS} with sharp constant) gives the following $L^1$-Sobolev inequality
\beq\label{S9'}\|f\|_{\ff n{n-\aa/2}}  \le C \int_{\R^n\times \R^n} \ff{|f(x)-f(y)|}{|x-y|^{n+\aa/2}}\,\d x\,\d y,\quad f\in W^{\aa/2,1}(\R^n)\end{equation}  for some constant $C>0$.
 The proof of \eqref{S9'} addressed in \cite{MS,FS} relies on the Hardy inequality for fractional Sobolev spaces. 
 We note that the Sobolev embedding theorems involving the spaces
$W_{\alpha/2,p}$ also can be obtained by interpolation techniques and by passing through Besov spaces, see for example \cite{Besov1, Besov2}. For the treatment of fractional Sobolev-type inequalities
we can refer to \cite{Ad, BCLS, Ta, NPV} and the references therein. 

 According to Theorem \ref{T2.0}(1) below, \eqref{S9'} holds   if and only if
\beq\label{S9''} \kk:=\inf_{\mu(A)\in (0,\infty)} \ff 1 {\mu(A)^{\ff{n-\aa/2}n}} \int_{A\times A^c} \ff{\d x \, \d y}{|x-y|^{n+\aa/2}}>0,\end{equation}
and furthermore, $\kk\in [\ff 1 {2C}, \ff{n}{2C(n-\aa/2)}],$ where $C$ is the sharp constant in \eqref{S9'}. Due to this fact, we also call \eqref{S9'} a Sobolev type isoperimetric inequality.

\

In this paper, we aim to establish isoperimetric inequalities for the following non-local form on a $\si$-finite measure space $(E,\F,\mu)$:
\beq\label{Form}  \E(f,f):= \frac{1}{2}\int_{E\times E} (f(x)-f(y))^2 \, J(\d x,\d y),\end{equation}
where $J(\d x, \d y)$ is a non-negative symmetric measure on $E\times E$.

Instead of the fractional Hardy  inequality used in \cite{MS,FS} and the Besov or
interpolation spaces used in \cite{Besov1, Besov2}, 
in our paper we will apply the super Poincar\'e inequality of $\E$, which was introduced by the first author in \cite{W00a}. This inequality can be regarded as a deformation of the Nash-type inequality, but is easier to verify in applications.
The proof here is self-contained.

We also mentioned that isoperimetric inequalities for symmetric diffusions have already been studied in the literature, see \cite{Bus, Ledoux, W00a, BCR1, BCR2} and the references therein. In particular, Ledoux's approach of Buser's inequality was used in \cite{W00a, BCR1} to illustrate the relation of the super-Poincar\'e inequalities with isoperimetry.
A notion of Orlicz hypercontractive semigroups was introduced in \cite{BCR1}, and their relations with various functional inequalities were studied. A measure-Capacity sufficient condition, in the spirit of Maz'ja \cite{Ma}, was established for super-Poincar\'e inequality
inequality in \cite{BCR2}. In the present setting, we are concerned with non-local forms. We will directly derive the equivalence of $L^1$ Orlicz-Sobolev inequality (involving the
$L^1$-norm of the jumping kernel for non-local forms) and  $L^1$-Poincar\'e type inequality, and also characterize the relationship between  the isoperimetric  inequality and  the super Poincar\'e inequality. In particular, one of our general results (see Theorem \ref{T2.1} below) implies the following Orlicz-Sobolev type isoperimetric inequality  \eqref{ON} on $\R^n$.

Following \cite[Section 1.3]{RR}, a function $N:[0,\infty)\to [0,\infty]$
 is called a Young function if it is convex and increasing with $N(0)=0$ and $N(\infty):= \lim_{s\to\infty} N(s)=\infty.$ We consider the following Orlicz norm induced by $N$ (see \cite[Section 3.2]{RR}):
$$\|f\|_{N}:=\inf\bigg\{r>0: \int_{\R^n} N\Big(\ff{|f(x)|} r\Big)\, \d x \le 1\bigg\},$$ where $\inf \emptyset =\infty$ by convention.
Let $L_N(\R^n)=\{f\in \B(\R^n): \|f\|_N<\infty\}.$ It is easy to see from the  convexity  and $N(0)=0$ that
$cN(s)\le N(cs)$ for $c\ge 1$. So, $N(\infty)>0$ is equivalent to $N(\infty)=\infty$, and
\beq\label{*WH} c^{-1}\|f\|_{cN}\le \|f\|_{N}\le \|f\|_{cN},\ \ c\ge 1.\end{equation}   For two Young functions   $N_1$ and $N_2,$  we say that $N_1$ is not dominated by $N_2$ if
$\sup_{s>0}\ff{N_1(s)}{N_2(s)}=\infty,$ where we set $\ff 0 0 =1$,
 $\ff \infty \infty=1$,
$\ff r 0 =\infty$ and $\ff r \infty =0$ for $r>0$. In this case, we write $N_1\npreceq N_2$.

For $\aa\in (0,2)$, let $\H_\aa$ be the class of functions $h: (0,\infty)\to (0,\infty)$ satisfying
\beg{enumerate} \item[(i)] $h(s)$ and $s{h(s)}^{-1}$ are increasing in $s$.
\item[(ii)] For any $s>0$, $$\Phi_h(s):= \int_0^s \d t \int_0^{t^{-1/n}} \ff{r^{\aa-1}}{h(r)}\, \d r<\infty.$$\end{enumerate}
It is easy to see that $\Phi_h$ is continuous, strictly increasing and concave with $\Phi_h(0)=0$.
Thus, $N_h:=\Phi_h^{-1}$ is a Young function.

\beg{thm}\label{T1.1} For any $\aa\in (0,2)$ and $h\in \H_\aa$, there exists a constant $C>0$ such that
\beq\label{ON} \|f\|_{N_h} \le C \int_{\R^n\times \R^n} |f(x)-f(y)|\ff{ h(|x-y|)}{\,\,|x-y|^{n+\aa}}\, \d x\, \d y,\ \ f\in L_{N_h}(\R^n),\end{equation} which  implies
\beq\label{ON'} \inf_{\mu(A)\in (0,\infty)}\bigg( N_h^{-1}(\mu(A)^{-1}) \int_{A\times A^c} \ff{h(|x-y|)}{\,\,|x-y|^{n+\aa}}\,\d x\, \d y\bigg)>0.\end{equation}
Consequently: \beg{enumerate} \item[$(1)$] For any $\aa_1,\aa_2\in (0,2)$, let
$$N^{\land}_{\aa_1,\aa_2}(s):= s^{\ff n {n-\aa_1/2}}\land s^{\ff n{n-\aa_2/2}},\ \ N^{\lor}_{\aa_1,\aa_2}(s):= s^{\ff n {n-\aa_1/2}}\lor s^{\ff n{n-\aa_2/2}},\ \ s\ge 0.$$
Then there exists a constant $C>0$ such that
\beq\label{N1} \|f\|_{N^{\land}_{\aa_1,\aa_2}} \le C \int_{\R^n\times \R^n} \ff{|f(x)-f(y)| }{|x-y|^{n+\aa_1/2}\lor |x-y|^{n+\aa_2/2}} \,\d x \,\d y,\ \ f\in L_{N^{\land}_{\aa_1,\aa_2}}(\R^n),\end{equation}
\beq\label{N2} \|f\|_{N^{\lor}_{\aa_1,\aa_2}} \le C \int_{\R^n\times \R^n} \ff{|f(x)-f(y)| }{|x-y|^{n+\aa_1/2}\land |x-y|^{n+\aa_2/2}} \,\d x \,\d y,\ \ f\in L_{N^{\lor}_{\aa_1,\aa_2}}(\R^n).\end{equation}
These inequalities are sharp in the sense that $\eqref{N1}\ ($resp.  $\eqref{N2} )$ fails if $N^{\land}_{\aa_1,\aa_2}\ ($resp. $N^{\lor}_{\aa_1,\aa_2} ) $ is replaced by a Young function $N\npreceq N^{\land}_{\aa_1,\aa_2} ($resp. $N^{\lor}_{\aa_1,\aa_2} ) $.
\item[$(2)$] For any $\aa\in (0,2)$ and $q, p\in \R$, let $\lambda\ge 2$ large enough such that both $N_{\aa}^{log,q,+}(s):= \big\{s\log^q\left(\lambda+ s\right)\big\}^{\ff n{n-\aa/2}}$ and
$N_{\aa}^{log,p,-}(s):= \big\{s\log^p\left(\lambda+ s^{-1}\right)\big\}^{\ff n{n-\aa/2}}$ are Young functions. Then there exists a constant $C>0$ such that for all $f\in L_{N^{log,q,+}_{\aa}}(\R^n),$
\beq\label{*21} \|f\|_{N_\aa^{log,q,+}}  \le C \int_{\R^n\times \R^n}|f(x)-f(y)| \ff{ \{\log(2+|x-y|^{-1})\}^q}{|x-y|^{n+\aa/2}} \,\d x \,\d y,\end{equation} and for all $f\in L_{N^{log,p,-}_{\aa}}(\R^n),$
\beq\label{*22} \|f\|_{N_\aa^{log,p,-}}  \le C \int_{\R^n\times \R^n}|f(x)-f(y)| \ff{ \{\log(2+|x-y|)\}^p}{|x-y|^{n+\aa/2}} \,\d x \,\d y.\end{equation}
These inequalities are sharp in the sense that $\eqref{*21}$ $($resp. $\eqref{*22})$ fails if $N_\aa^{log,q,+}$  $($resp. $N_\aa^{log,p,-}) $ is replaced by a Young function $N\npreceq N_\aa^{log,q,+} ($resp. $N_\aa^{log,p,-} )$. \end{enumerate}
\end{thm}

According to Theorem \ref{T2.0}(1) below, \eqref{ON} and \eqref{ON'} are equivalent in more general case,
so  an $L^1$ Orlicz-Sobolev inequality of type \eqref{ON} is also called an Orlicz-Sobolev type isoperimetric inequality.
It is easy to see that when $\aa_1=\aa_2=\aa$ and $q=p=0$, the inequalities \eqref{N1}, \eqref{N2}, \eqref{*21} and \eqref{*22} coincide with \eqref{S9'}.
The Orlicz-Sobolev type isoperimetric inequalities \eqref{N1}--\eqref{*22} are equivalent to the corresponding  Poincar\'e type ones,  see Corollary \ref{C4.1} for details.

\ \

In the remainder of the paper, we will work with the form   \eqref{Form} under a general framework. In Section 2, we characterize the  link between the super poincar\'e and  isoperimteric inequalities. In Section 3, we first apply the main result derived in Section 2 to prove Theorem \ref{T1.1}, then make extensions to    the truncated and discrete $\aa$-stable Dirichlet forms.   Finally, by using a perturbation argument, we  derive  isoperimetric inequalities in Section 4   for $\aa$-stable-like Dirichlet forms with finite reference measures.

\section{Super Poincar\'{e} and isoperimetric inequalities: general results}

Let $(E,\F,\mu)$ be a $\si$-finite measure space,
and let $J(\d x,\d y)$ be a non-negative and symmetric measure on $E\times E$.
In this section, we investigate the link between the isoperimetric inequality and the  super Poincar\'e  inequality for the following symmetric quadratic form
\beq\label{DR} \beg{split} &\E(f,g):=\frac{1}{2}\int_{E\times E} (f(x)-f(y))(g(x)-g(y)) J(\d x,\d y),\\
& f,g\in\D(\E):= \big\{f\in L^2(\mu): \E(f,f)<\infty\big\}.\end{split}\end{equation} To ensure that $\E(f,f)$ does not depend on the choice of $\mu$-versions of $f$, we assume that $J(\d x,\d y)= J(x,y)\,\mu(\d x)\,\mu(\d y)$ for some symmetric density $J: E\times E\to [0,\infty)$. Moreover, we assume that $\D(\E)$ is dense in $L^2(\mu)$ so that $(\E,\D(\E))$ is a symmetric Dirichlet form.

According to \cite{W00a}, we say that $(\E, \D(\E))$ satisfies the super Poincar\'e inequality with rate function $\bb: (0,\infty)\to (0,\infty)$, if
\begin{equation}\label{SP1}\|f\|_2^2 \le r \E(f,f)+ \beta(r)\|f\|_1^2,\quad f\in \D(\E), r>0.\end{equation} Here and in what follows, for any $p\in[1,\infty]$, $\|\cdot\|_p$ denotes the $L^p$-norm with respect to $\mu$.  Since $\E(f,f)\ge 0,$  we may and do assume that $\bb$ is decreasing on $(0,\infty)$. See   \cite{W00a,W00b,Wbook} and references within for the super Poincar\'e inequality and applications.

  For a Young function  $N$,   let $\|\cdot\|_{N}$ be the Orlicz norm induced by $N$ and the measure $\mu$, and let  $L_N(\mu)= \{f\in\B(E): \|f\|_{N}<\infty\},$ where $\B(E)$ is the class of measurable functions on $E$.
  The left derivative of Young function $N$, denoted by $N'_-$, always exists and is non-decreasing left continuous on $(0,\infty)$, see e.g.\ \cite[Section 1.3]{RR}. For any non-negative decreasing  function $f$ on $[0,\infty)$, let  $$f^{-1}(r):=\inf\{s>0: f(s)\le r\},\ \ r\ge0,$$  where $\inf\emptyset :=\infty$.
      Similarly, for any non-negative increasing  function $f$ on $[0,\infty)$, let  $$f^{-1}(r):=\inf\{s>0: f(s)\ge r\},\ \ r\ge0.$$

In the following four subsections, we first observe the equivalence of an $L^1$ functional inequality and the corresponding isoperimetric inequality,
then investigate the link between the super Poincar\'e   and  isoperimetric inequalities, and finally extend the main results to the case with killing.

  \subsection{  $L^1$ functional   and isoperimetric inequalities}

  Consider the $L^1$ Orlicz-Sobolev inequality
    \beq\label{LOS} \|f\|_N\le C \int_{E\times E} |f(x)-f(y)|\, J(\d x, \d y),\ \ f\in L_N(\mu),\end{equation} and
  the $L^1$ Poincar\'e type inequality
  \beq\label{PI0} \|f\|_2^2 \le C_1 \int_{E\times E}|f^2(x)- f^2(y)| \, J(\d x,\d y) + C_2 \|f\|_1^2,\ \ f\in L^2(\mu),\end{equation}
  where $N$ is a Young function and $C, C_1, C_2>0$ are constants. The following result provides their equivalent isoperimetric inequalities.

  \beg{thm}\label{T2.0}\begin{itemize}
  \item[$(1)$] The inequality $\eqref{LOS}$ implies
  \beq\label{LOS'} \inf_{\mu(A)\in (0,\infty)} \big\{N^{-1}(\mu(A)^{-1}) J(A\times A^c)\big\}\ge \kk\end{equation} holds for $\kk=\ff 1 {2C}.$
  On the other hand, if $N'_-(s)>0$ for $s>0$ such that
  $$c_N:= \inf_{s>0} \ff{N(s)}{s N'_-(s)} >0,$$ then $\eqref{LOS'}$ implies $\eqref{LOS}$ for $C= \ff 1 {2 c_N\kk}.$

  \item[$(2)$]   The inequality $\eqref{PI0}$   implies
    \beq\label{PI0'} \mu(A)\le 2C_1 J(A\times A^c) +\tt C_2 \mu(A)^2,\ \ \mu(A)\in (0,\infty)\end{equation} for $\tt C_2= C_2$. On the other hand,
  $\eqref{PI0'}$ implies $\eqref{PI0}$ for $C_2= 2\tt C_2.$\end{itemize} \end{thm}

  \beg{proof} (1) For $A\subset E$ with $\mu(A)\in (0,\infty)$, let $f={\bf1}_A$. Then
 \beq\label{AW} \int_{E\times E} |f(x)-f(y)|\, J(\d x,\d y)= 2 J(A\times A^c).\end{equation}
  Moreover, for any $r>0$, by the definition of $N^{-1}$ we see that
  $$\int_E N\big(|f(x)|/r\big)  \,\mu(\d x)= N(r^{-1})\mu(A)\le 1$$ implies $r\ge N^{-1}(\mu(A)^{-1}).$ Therefore,
  $\|f\|_N\ge N^{-1}(\mu(A)^{-1}).$ Combining this with \eqref{LOS} and \eqref{AW}, we prove \eqref{LOS'} for $\kk= \ff 1 {2C}.$

    On the other hand, let $c_N>0$ and \eqref{LOS'} hold. It suffices to prove \eqref{LOS} for $C= \ff 1 {2c_N\kk}$ and any $f\ge 0$ with $\|f\|_N=1.$
    By Fubini's theorem and \eqref{LOS'}, we have
    \beq\label{*AB} \beg{split} &\int_{E\times E}|f(x)-f(y)|\, J(\d x,\d y)= 2 \int_{\{f(x)>f(y)\}} \bigg(\int_{f(y)}^{f(x)} \,\d r \bigg)\,J(\d x,\d y)\\
    &= 2 \int_0^\infty J(\{(x,y): f(x)> r\ge f(y)\}) \,\d r \ge 2 \kk \int_0^\infty \ff{\d r}{N^{-1}(\mu(f> r)^{-1})} \\
    &= 2\kk \int_0^\infty \ff{\mu(N(f)> N(r))}{\mu(N(f)> N(r)) N^{-1}(\mu(N(f)> N(r))^{-1})}\,\d r \\
    &= 2\kk \int_0^\infty \ff{\mu(N(f)> s)}{N'_-(N^{-1}(s)) \mu(N(f)> s)N^{-1}(\mu(N(f)> s)^{-1})}\,\d s.\end{split}\end{equation}
    Since $f\ge 0$  and $\|f\|_N=1$, we have
 $$ \mu(N(f)> s)\le \ff{\mu(N(f))}s =\ff 1 s,\ \ s>0.$$ Noting that $N'_-\circ N^{-1}$ is increasing,  by
 letting $t=N^{-1}(\mu(N(f)> s)^{-1})$ we obtain
 \beg{align*}&N'_-(N^{-1}(s)) \mu(N(f)> s) N^{-1}(\mu(N(f)> s)^{-1}) \\
 &\le N'_-(N^{-1}(\mu(N(f)> s)^{-1}))  \mu(N(f)> s) N^{-1}(\mu(N(f)> s)^{-1})\\
 &= \ff{N'_-(t)t}{N(t)}\le \ff 1 {c_N}.\end{align*} Substituting into \eqref{*AB} and noting that $\|f\|_N=1$, we arrive at
 $$\int_{E\times E} |f(x)-f(y)|\,J(\d x,\d y) \ge 2\kk c_N\mu(N(f))= 2\kk c_N= 2\kk c_N \|f\|_N.$$ Thus, \eqref{LOS} holds for $C= \ff 1 {2\kk c_N}.$

    (2) As in (1), by applying \eqref{PI0} to $f={\bf1}_A$ we prove \eqref{PI0'} for $\tt C_2=C_2$. On the other hand, let $f\in L^2(\mu)$ with $\|f\|_1=1.$ Then
    $\mu(f^2>s)\le s^{-1/2}$, so that, as in \eqref{*AB},   \eqref{PI0'} yields
    \beg{align*} &C_1 \int_{E\times E} |f^2(x)-f^2(y)| \,J(\d x,\d y) = 2 C_1 \int_0^\infty J(\{f^2> s\}\times \{f^2\le s\}) \,\d s \\
    &\ge \int_0^\infty \big\{\mu(f^2> s)- \tt C_2 \mu(f^2> s)^2\big\}\,\d s \ge \mu(f^2)- \tt C_2 \int_0^\infty \ff 1 {\ss s} \mu\big(|f|> \ss s\big)\, \d s\\
    &= \mu(f^2) - 2 \tt C_2 \int_0^\infty \mu\big(|f|>\ss s\big)\,\d\ss s= \mu(f^2)-2\tt C_2 \mu(|f|) = \mu(f^2)- 2\tt C_2 \mu(|f|)^2.\end{align*}
    Therefore, \eqref{PI0'} implies \eqref{PI0} for $C_2=2\tt C_2.$
  \end{proof}

\subsection{From super Poincar\'e to isoperimetric}

Let $P_t$ be the (sub-) Markov semigroup associated with the  symmetric Dirichlet form $(\E,\D(\E))$.

\beg{thm}\label{T2.1} Assume that $\eqref{SP1}$ holds with $\beta(\infty):=\lim_{r\to \infty}\bb(r)=0$. Let $\gg: E\times E\to [0,\infty)$ with $\gg(x,y)>0$ and $\gg(x,y)=\gg(y,x)$ for $x\ne y$,  and define
$$ \theta_\gg(t):= \sup_{\|g\|_\infty\le 1} {\rm ess}_{\mu\times\mu}\sup_{x\ne y} \ff{|P_tg(x)-P_tg(y)|}{\gg(x,y)},\ \ t>0.$$ If
\beg{align*}&\Phi_\gg(s):= \int_0^s\,\d r \int_0^{\bb^{-1}(r)} \theta_\gg(t) \,\d t <\infty,\ \ s>0,\end{align*}  then $N_\gg:=\Phi_\gg^{-1}$ is a Young function, and there exists a constant $C>0$ such that
\beq\label{T21}\|f\|_{N_\gg} \le C \int_{E\times E} |f(x)-f(y)| \gg(x,y)\,J(\d x, \d y),\ \  f\in L_{N_\gg}(\mu). \end{equation} \end{thm}

\ \

To prove this result, we consider the   symmetric measure
$$J_\gg(\d x,\d y):= \gg(x,y)\,J(\d x,\d y) $$ on $E\times E$, and introduce the isoperimetric constants
\begin{equation}\label{IC}\kk_\gg(s):= \inf\bigg\{\ff{J_\gg(A\times A^c)}{\mu(A)}:\ \mu(A)\in (0,s)\bigg\},\ \ s>0,\end{equation} where $\inf \emptyset:=\infty$.
We have the following result.

\begin{lem}\label{L1} For any increasing function $G: [0,\infty)\to [0,\infty)$ with $G(0)=0$ and $G(s)>0$ for $s>0$, it holds that
\begin{equation}\label{S01}\int_E \,\d\mu \int_0^{|f|} \kk_\gg\left( G(s)^{-1}\right) \,\d s\le  \frac{1}{2}\int_{E\times E} |f(x)-f(y)|\, J_\gg(\d x,\d y),\quad \mu(G(|f|))=1.\end{equation}
Consequently:
\begin{itemize}\item[$(1)$] If $\kk_\gg(s)>0$   for some $s>0,$ then
\beq\label{IS1} \|f\|_2^2\le \ff 1 {2\kk_\gg(s)} \int_{E\times E} |f^2(x)-f^2(y)| \,  J_\gg (\d x,\d y)  + \ff 2 s \|f\|_1^2,\quad f\in L^2(\mu).\end{equation}

\item[$(2)$] If $\kk_\gg(s)>0$ for all $s>0$ such that $$\Phi(t):= \int_0^t \ff{\d r}{\kk_\gg(r^{-1})}<\infty, \ \ t> 0,$$  then $N:= \Phi^{-1}$ is a Young function,
and
\beq\label{S03}\|f\|_N \le \frac{1}{2}\int_{E\times E} |f(x)-f(y)| \, J_\gg(\d x,\d y),\ \ f\in L_N(\mu). \end{equation} \end{itemize}
  \end{lem}

\begin{proof}    For any  $f\in\B(E)$ with $\mu(G(|f|))=1$, we have
 $$\mu(|f|>s)\le G(s)^{-1},\quad s>0.$$ As in \eqref{*AB}, this and the definition of $\kk_\gg(s)$ imply that
\beg{align*} &\frac{1}{2} \int_{E\times E}  |f(x)-f(y)|\,  J_\gg(\d x,\d y)\ge \frac{1}{2} \int_{E\times E}  \big||f|(x)-|f|(y)\big|\,  J_\gg(\d x,\d y)\\
&=  \int_0^\infty   J_\gg(\{|f|>u\}\times  \{|f|\le u \} )\,\d u\ge  \int_0^\infty \kk_\gg\left( G(u)^{-1}\right)\mu(|f|> u)\,\d u\\
&= \int_E \,\d\mu \int_0^{|f|} \kk_\gg\left( G(u)^{-1}\right) \,\d u.\end{align*}
We have proved \eqref{S01}.
Below we prove assertions (1) and (2) respectively.

{\bf Assertion  (1)}.  For any $f\in \B(E)$ with $\|f\|_1=1$, we have
\beq\label{UU} \mu(|f|> s)\le \ff 1 s,\quad s>0.\end{equation} Since $\kk_\gg(s)$ is decreasing in $s$,    applying  \eqref{S01} to $f^2$ with $G(s)=s^{1/2}$, we derive
\beq\label{*9}\beg{split}& \frac{1}{2} \int_{E\times E}  |f^2(x)-f^2(y)|\,  J_\gg(\d x,\d y)
\ge \int_E \,\d\mu \int_0^{f^2} \kk_\gg\left( u^{-1/2}\right) \,\d u\\
&= \int_0^\infty \kk_\gg\left( u^{-1/2}\right)\mu(f^2> u)\,\d u
 \ge  \kk_\gg\left(s\right)\int_{s^{-2}}^\infty \mu(f^2>u)\,\d u,\quad s>0.\end{split}  \end{equation}
On the other hand, by \eqref{UU}  we have
$$\int_0^{s^{-2}} \mu(f^2>u)\,\d u \le \int_0^{s^{-2}} \ff 1 {\Ss u}\,\d u= \ff 2 s,\quad s>0.$$ Combining this with \eqref{*9} and  $$\|f\|_2^2 = \int_0^\infty \mu(f^2>u)\,\d u,$$ we prove \eqref{IS1}.

{\bf Assertion (2)}.   Let $N=\Phi^{-1}$. Then $N$ satisfies $N(0)=0$, and solves the equation
\beq\label{GE} \frac{\d N(s)}{\d s}= \kk_\gg\big(N(s)^{-1}\big)\quad \ {\rm a.e.}\ s>0,\end{equation} where $\frac{\d N(s)}{\d s}$ denotes the Radon-Nikodym derivative of $N$ with respect to the Lebesgue measure.
Since $\kk_\gg(s)$ is strictly positive and decreasing in $s$, and since $\Phi(t)<\infty$ for $t>0$, it is easy to deduce from \eqref{GE} that $N$ is a Young function, and
$$\int_0^{|f|} \kk_\gg\big(N(s)^{-1}\big) \,\d s =\int_0^{|f|} \frac{\d N(s)}{\d s}\,\d s= N(|f|).$$ Combining this with  \eqref{S01} leads to
$$1\le \frac{1}{2}\int_{E\times E} |f(x)-f(y)|\,  J_\gg (\d x,\d y),\ \ \mu(N(|f|))=1,$$
which in turn implies \eqref{S03}.
\end{proof}

According to Lemma \ref{L1}, for the proof of Theorem \ref{T2.1}  we only need to estimate the isoperimetric constants $\kk_\gg(s)$ using \eqref{SP1}.
The following result can be regarded as an extension of a
result of \cite{Bus} (see also \cite{Ledoux}) to non-local forms.

\begin{lem} \label{L2} Let $\gg$ and $\theta_\gg(t)$ be in Theorem $\ref{T2.1}. $  If
$$\Theta_\gg(t) := \int_0^t \theta_\gg(s)\,\d s<\infty,\ \ t\ge 0,$$
then the super Poincar\'e inequality \eqref{SP1} implies
\beq\label{PL} \kk_\gg (s)\ge  \frac{1-\e^{-1}}{2\Theta_\gg(\beta^{-1}(1/({2s})))},\quad s>0.\end{equation}
 \end{lem}
\begin{proof} For any $f,g\in \D(\E)$, we have
\begin{align*}\mu(g(f-P_tf)) &= \mu(f(g-P_tg)) = \int_0^t \E(f,P_s g)\,\d s\\
&= \ff 1 2 \int_0^t\d s \int_{E\times E} (P_s g(x)-P_s g(y)) (f(x)-f(y))\, J (\d x,\d y)\\
&\le\frac{\|g\|_\infty}{2}  \left(\int_0^t\theta_\gg(s)\,\d s\right)
  \int_{E\times  E} |f(y)-f(x)|\, J_\gg (\d x,\d y) \\
&= \frac{\Theta_\gg(t)\|g\|_\infty}{2} \int_{E\times  E} |f(y)-f(x)|\, J_\gg (\d x,\d y).\end{align*}
Thus,
\begin{equation}\label{e:proof-1}\mu(|f-P_tf|)\le \frac{\Theta_\gg(t)}{2} \int_{E\times  E} |f(y)-f(x)|\,  J_\gg(\d x,\d y).\end{equation}

Next, by \cite[(3.4)]{W00a},   the super Poincar\'{e} inequality \eqref{SP1} is equivalent to
$$\|P_tf\|_2^2\le \|f\|_2^2 \exp(-2t/r)+\beta(r)\|f\|_1^2\left(1-\exp(-2t/r)\right), \quad t,r>0.$$
In particular, for any $A\Subset E$ with $\mu(A)<\infty$, we have
\begin{equation}\label{e:posup-1}\|P_{t/2} {\bf 1}_A\|_2^2 \le \mu(A)\exp(-t/r)+ \mu(A)^2 \beta(r)\left(1-\exp(-t/r)\right), \quad t,r>0.\end{equation}

Now,  for $s>0$ and $A\subset E$ with $0<\mu(A)<s$,
 \eqref{e:proof-1}   gives
\beq\label{*0}\mu(|{\bf 1}_A-P_t{\bf 1}_A|)\le \frac{\Theta_\gg(t)}{2} \int_{E\times  E} |{\bf 1}_A(y)-{\bf 1}_A(x)|\,  J_\gg(\d x,\d y)=\Theta_\gg(t)   J_\gg(A\times  A^c).\end{equation}
On the other hand, we have
\begin{equation*}\beg{split}   \mu (|{\bf 1}_A- P_t {\bf 1}_A|)=&\int_A (1-P_t {\bf 1}_A)\,\d \mu+ \int_{A^c}P_t {\bf 1}_A\,\d\mu \\
\ge &\int_A (1-P_t {\bf 1}_A)\,\d \mu  =\mu(A)-\int_A P_t {\bf 1}_A\,\d \mu =\mu(A)- \|P_{t/2} {\bf 1}_A\|_2^2.
\end{split} \end{equation*}
This together with \eqref{e:posup-1} yields that for any $t,r>0$,
\begin{align*}\mu(|{\bf 1}_A-P_t{\bf 1}_A|)\ge & \mu(A)\big(1-\mu(A)\beta(r)\big)\big(1-\exp(-t/r)\big).\end{align*} Taking $r=t=\beta^{-1}(1/(2\mu(A)))$ in the inequality above, we get that
$$\mu(|{\bf 1}_A-P_t{\bf 1}_A|)\ge\frac{1-\e^{-1}}{2} \mu(A).$$
Combining this with \eqref{*0}   we arrive at
$$\frac{  J_\gg (A\times  A^c)}{\mu(A)}\ge  \frac{1-\e^{-1}}{2\Theta_\gg(\beta^{-1}(1/(2\mu(A))))}\ge \frac{1-\e^{-1}}{2\Theta_\gg(\beta^{-1}(1/(2s)))},$$ where in the last inequality we have used the facts that $0<\mu(A)<s$, $\beta$ is decreasing and $\Theta_\gg$ is increasing. Therefore, \eqref{PL} holds.
 \end{proof}

\ \

\beg{proof}[Proof of Theorem $\ref{T2.1}$] Let $\Phi_\gg$ be in Theorem \ref{T2.1}. Since $\bb(s)$ is strictly positive and deceasing on $(0,\infty)$, it is easy to see that $N_\gg:=\Phi_\gg^{-1}$ is a Young function. Since $\bb(\infty)=0$, by Lemma \ref{L2} we have $\kappa_\gg(s)>0$ for all $s>0$, and
\beg{align*} \Phi(t):&= \int_0^t \ff{\d u}{\kk_\gg(u^{-1})} \le \ff 2 {1-e^{-1}} \int_0^t \Theta_\gg\big(\bb^{-1}(u/2)\big)\, \d u\\
&\le \ff 4 {1-\e^{-1}} \int_0^t \Theta_\gg\big(\bb^{-1}(r)\big)\,\d r = \ff 4 {1-\e^{-1}} \Phi_\gg(t),\ \ t\ge 0.\end{align*}
Thus, $N_\gg(s):= \Phi_\gg^{-1}(s)\le \Phi^{-1}(4s/(1-\e^{-1})):=N(4s/(1-\e^{-1})).$ Combining this with \eqref{S03} and \eqref{*WH}, we prove \eqref{T21}.
\end{proof}

\subsection{From isoperimetric to super Poincar\'e }

Let $(\E,\D(\E))$ be given by \eqref{DR}. For a non-negative symmetric function $\gg$ on $E\times E$, let $J_\gg(\d x,\d y):= \gg(x,y)\, J(\d x,\d y)$, and $\kk_\gg(s)$ be the isoperimetric constant defined by \eqref{IC}. For a Young function $N$, we aim to deduce the super Poincar\'e inequality \eqref{SP1} from the Orlicz-Sobolev type isoperimetric inequality
\beq\label{OS} \|f\|_N \le C\int_{E\times E} |f(x)-f(y)|\,J_\gg(\d x,\d y),\ \ f\in L_N(\mu).\end{equation} To this end, we also consider the Poincar\'e type isoperimetric inequality
\beq\label{PI} \|f\|_2^2\le r \int_{E\times E} |f^2(x)-f^2(y)| \,J_\gg(\d x,\d y) + \bb_1(r) \|f\|_1^2,\  \ r>0, f\in L^2(\mu)\end{equation} for some decreasing function $\bb_1: (0,\infty) \to (0,\infty).$

\beg{thm}\label{T4.1} Assume $\eqref{OS}$   for some Young function $N$ such that $s\mapsto s^{-1}N(s)$ is increasing on $(0,\infty)$. Then:
\beg{enumerate} \item[$(1)$] For any $s>0$,
$$\kk_\gg(s)\ge \dfrac 1 {2Cs N^{-1}(s^{-1})}.$$
\item[$(2)$] $\eqref{PI}$ holds with
$$\bb_1(r):= 2 \inf\big\{s>0:\ Cs^{-1}N^{-1}(s)\le r\big\}, \ \ r>0.$$
\item[$(3)$] If the density $J(x,y):= \ff{J(\d x,\d y)}{\mu(\d x)\mu(\d y)} $ and $\gg$ satisfy
 \begin{equation}\label{con}c_\gg:={\rm ess}_\mu\sup_x\int_E \gg(x,y)^2 J(x,y)\,\mu(\d y)<\infty,\end{equation}
    then $\eqref{SP1}$ holds with
$$\bb(r):= 4 \inf\bigg\{s>0:\ s^{-1}N^{-1}(s)\le \ff{\ss r}{2C\sqrt{2c_\gg}}\bigg\}, \ \ r>0.$$
\end{enumerate} \end{thm}

\beg{proof} For any $s>0$ and $A\subset E$ with $\mu(A)\in (0,s)$, take $f= N^{-1}(\mu(A)^{-1})\1_A$. Then $\|f\|_N=1$  and  due to \eqref{OS},
$$1\le 2C J_\gg(A\times A^c) N^{-1}(\mu(A)^{-1}).$$ Therefore,
$$\kk_\gg(s) \ge \ff 1 {2C} \inf_{r\in (0,s)} \ff 1 {r N^{-1}(r^{-1})}.$$
Since $s N^{-1}(s^{-1})$ is increasing in $s>0$, this implies (1).

It is easy to see that (2) follows from (1) and Lemma \ref{L1}(1).  It remains to prove (3). By (1), Lemma \ref{L1}(1), and the Cauchy-Schwarz inequality, we obtain
\beg{align*} \|f\|_2^2&\le C s N^{-1}(s^{-1}) \int_{E\times E } |f^2(x)-f^2(y)| \, J_\gg(\d x,\d y)+ \ff 2 s \|f\|_1^2\\
&\le C s N^{-1}(s^{-1}) \left(\int_{E\times E } (f(x)-f(y))^2 \, J(\d x,\d y)\right)^{1/2} \\
&\qquad\qquad\qquad\quad \times \left(\int_{E\times E } (f(x)+f(y))^2 \gg(x,y)^2\, J(\d x,\d y)\right)^{1/2} + \ff 2 s \|f\|_1^2\\
&\le  2C \sqrt{2c_\gg} s N^{-1}(s^{-1})\sqrt{\E(f,f)\|f\|_2^2}+  \ff 2 s \|f\|_1^2\\
&\le \ff 1 2 \|f\|_2^2+ 4 C^2c_\gg (s N^{-1} (s^{-1}))^2 \E(f,f) + \ff 2 s \|f\|_1^2 ,\ \ s>0,\end{align*} where in the third inequality we have used \eqref{con}. This implies \eqref{SP1} for the desired $\bb$. \end{proof}

Similarly, we have the following result.
\beg{thm}\label{T4.2} Assume that $\eqref{PI}$ holds with $\bb_1(\infty):=\lim_{r\to\infty}\bb_1(r)=0$. Then:
\beg{enumerate} \item[$(1)$] For any $s>0$,
$$ \kk_\gg(s)\ge \ff 1 {4 \bb_1^{-1}(1/(2s))}.$$
\item[$(2)$] If $$\Phi(t):= 4\int_0^t \bb_1^{-1} (r/2) \,\d r<\infty, \ \ t>0,$$ then $\eqref{OS}$ holds with $N:=\Phi^{-1}$.
\item[$(3)$]   \eqref{con} implies $\eqref{SP1}$  with
$$\bb(r):=  2  \bb_1\left(\ss r/(2\sqrt{2 c_\gg})\right), \ \ r>0.$$
\end{enumerate} \end{thm}
\begin{proof} By Theorem \ref{T2.0}(2),   \eqref{PI} implies
$$\mu(A)\le 2rJ_\gg(A\times A^c)+\beta_1(r)\mu(A)^2,\ \ r>0.$$
Thus, $$1\le 2r \frac{J_\gg(A\times A^c)}{\mu(A)}+\beta_1(r)\mu(A)\le 2r \frac{J_\gg(A\times A^c)}{\mu(A)}+\beta_1(r)s,\ \ r>0.$$Taking $r=\beta_1^{-1}((2s)^{-1})$ in the inequality above, we get that
$$\frac{J_\gg(A\times A^c)}{\mu(A)}\ge \ff 1{4\beta_1^{-1}(1/(2s))}.$$ This implies (1).

(2) immediately follows from (1) and Lemma \ref{L1}(2), and  (3) can be proved by the   argument for Theorem \ref{T4.1}(3).
\end{proof}

As a consequence of Theorem  \ref{T4.1}(2) and Theorem \ref{T4.2}(2), we have the following correspondence of \eqref{OS} and \eqref{PI}.

\beg{cor}\label{C4.1} Let $p_1, p_2, p>1$ and $q\in \R$ be constants. Then,
\beg{enumerate}
 \item[$(1)$] $\eqref{OS}$ holds with $N(s)= s^{p_1}\land s^{p_2} $   if and only if $\eqref{PI}$ holds with
$$\bb_1(r):= c (r^{-\ff {p_1}{ p_1-1}} \lor r^{-\ff {p_2}{p_2-1}}),\ \ r>0 $$ for some constant $c>0$.
\item[$(2)$] $\eqref{OS}$ holds with $N(s)= s^{p_1}\lor s^{p_2} $   if and only if $\eqref{PI}$ holds with
$$\bb_1(r):= c (r^{-\ff {p_1}{ p_1-1}} \land  r^{-\ff {p_2}{p_2-1}}),\ \ r>0 $$ for some constant $c>0$.
\item[$(3)$] Let $\lambda\ge2$ such that $N(s):= s^p\{\log(\ll+s^{-1})\}^q$  is Young function  and $s\mapsto s^{-1}N(s)$ is increasing on $(0,\infty)$. Then, $\eqref{OS}$ holds with $N(s)$ if and only if $\eqref{PI}$ holds with
$$\bb_1(r):= c r^{-\ff p{ p-1}} \{\log(2+r)\}^{-\ff q {p-1}},\ \ r>0 $$ for some constant $c>0$.

\item[$(4)$] Let $\ll\ge2$ such that $N(s):= s^p\{\log(\ll+s)\}^q$  is Young function  and $s\mapsto s^{-1}N(s)$ is increasing on $(0,\infty)$. Then, $\eqref{OS}$ holds with $N(s)$ if and only if $\eqref{PI}$ holds with
$$\bb_1(r):= c r^{-\ff p{ p-1}} \{\log(2+r^{-1})\}^{-\ff q {p-1}},\ \ r>0 $$ for some constant $c>0$.
\end{enumerate}
\end{cor}

\subsection{Extension to the case with killing}

We will add a potential term to the Dirichlet form $(\E,\D(\E))$ given in \eqref{DR}. Let $V$ be a non-negative measurable function on $(E,\F)$ such that the class
$$\D(\E_V):= \bigg\{f\in L^2(\mu): \E_V(f,f):= \E(f,f)+ \int_E f^2\,V(\d x)<\infty\bigg\}$$ is dense in $L^2(\mu)$, where $V(\d x):= V(x)\,\mu(\d x).$ Then $(\E_V, \D(\E_V))$ is a Shr\"odinger type symmetric energy form in $L^2(\mu)$, where
\begin{equation}\label{nondi}\E_V(f,g):= \E(f,g) +\int_E  f(x)g(x)\, V(\d x),\ \ f,g\in \D(\E_V). \end{equation}

It is standard that by enlarging the state space we are able to reduce to present setting to the case without killing, see \cite{LS,CW00}. More precisely, let  $ \bar E := E\cup \{\DD\}$ for an additional state $\DD$, and define \beg{align*} &\bar\mu(\d x)= \1_E(x)\,\mu(\d x)+ \dd_\DD(\d x),\\
 &\bar J(\d x,\d y)=
 \1_{E\times E} J(\d x, \d y)+\1_{\{\DD\}\times E}(x,y) \,\dd_\DD(\d x)\,V(\d y)+ \1_{E\times\{\DD\}}(x,y)\,V(\d x)\,\dd_\DD(\d y),\end{align*}
 where $\dd_\DD$ is the Dirac measure at point $\DD$.
 Since $J(\d x,\d y)= J(x,y)\,\mu(\d x)\,\mu(\d y)$, we have
 $$\bar J(\d x,\d y)= \bar J(x,y)\,\bar\mu(\d x)\,\bar \mu(\d y),$$ where
 $$\bar J(x,y):= \1_{E\times E} (x,y) J(x,y)+ \1_{\{\DD\}\times E}(x,y)V(y)+ \1_{E\times \{\DD\}}(x,y)V(x).$$

 Next, for a non-negative symmetric function $\gg$ on $E\times E$ and a non-negative function $\xi$ on $E$,  let
\beg{align*} &\bar\gg(x,y)= \1_{E\times E} (x,y)\gg(x,y) + \1_{\{\DD\}\times E} (x,y) \xi(y)+ \1_{E\times \{\DD\}}(x,y)\xi(x).\end{align*}
 Then  for any $x\in E$,
$$\int_{\bar E} \bar\gg(x,y)^2 \bar J(x,y)\,\bar\mu(\d y)= \int_{ E} \gg(x,y)^2 J(x,y)\,\mu(\d y) +\xi(x)^2V(x).$$

 Finally, for any measurable function $f$ on $E$, we extend it into $\bar f$ defined on $\bar E$ and by letting $\bar f(\DD)=0$. Then
 \beq\label{GF} \E_V(f,g) = \bar \E(\bar f,\bar g):=\ff 1 2 \int_{\bar E\times \bar E} (\bar f(x)-\bar f(y))(\bar g(x)-\bar g(y))\,\bar J(\d x,\d y),\ \ f,g\in \D(\E_V).\end{equation}
Let $P_t^V$ be the (sub)-Markov semigroup on $L^2(\mu)$ associated to $(\E_V,\D(\E_V))$, while $\bar P_t$ is the corresponding semigroup on $L^2(\bar\mu)$.  We have
\beg{align*} \bar\theta(t)&:= \sup_{\|\bar g\|_\infty\le 1} {\rm ess}_{\bar\mu\times\bar\mu}\sup_{x\ne y} \ff{|\bar P_t\bar g(x)-\bar P_t\bar g(y)|}{\bar \gg(x,y)}\\
&= \sup_{\| g\|_\infty\le 1} {\rm ess}_{ \mu\times \mu}\sup_{x\ne y} \max\bigg\{\ff{| P_t  g(x)- P_t  g(y)|}{ \gg(x,y)},\ \ff{|P_t g(x)|}{\xi(x)}\bigg\} ,\ \ t>0.\end{align*}

With the aid of all the notations above, by applying Theorem \ref{T2.1} and Theorem \ref{T4.1} to $\bar \E$ and $\bar\mu$ we obtain the following result.

\beg{thm}\label{T4.3}   Suppose that the super Poincar\'e inequality $\eqref{SP1}$ holds for $(\E_V,\D(\E_V))$ replacing $(\E,\D(\E))$ with some decreasing function $\bb:(0,\infty)\to (0,\infty)$ satisfying that $\beta(\infty)=0$. If
$$\bar \Phi_\gg(s):= \int_0^s\,\d r \int_0^{\bb^{-1}(r)} \bar \theta(t)\, \d t <\infty,\ \ s>0,$$ then $\bar N_\gg:=\bar\Phi_\gg^{-1}$ is a Young function, and there exists a constant $C>0$ such that
\beq\label{NB} \|f\|_{\bar N_\gg} \le C \bigg(\int_{E\times E} |f(x)-f(y)|\gg(x,y)\, J(\d x,\d y)+ \int_E |f(x)| \xi(x)\, V(\d x)\bigg) \end{equation} holds for all $f\in L_{\bar N_\gg } (\mu).$

 On the other hand, suppose that
 $$\sup_{x\in E}\bigg( \int_{ E} \gg(x,y)^2 J(x,y)\,\mu(\d y)+\xi(x)^2V(x)\bigg)<\infty.$$  If $\eqref{NB}$ holds for some Young function $N$ replacing $\bar N_\gg$ and satisfying that $s\mapsto s^{-1}N(s)$ is increasing on $(0,\infty)$, then there exist   constants $c_1,c_2>0$ such that $\eqref{SP1}$ holds for $(\E_V,\D(\E_V))$ replacing $(\E,\D(\E))$ with
$$\bb(r):= c_1 \inf\Big\{s>0:\ s^{-1}N^{-1}(s)\le c_2\ss r\Big\}, \ \ r>0.$$
\end{thm}

\section{Proof of Theorem \ref{T1.1} and   extensions}

\subsection{  Proof of Theorem \ref{T1.1}}
By Theorem \ref{T2.0}(1), it suffices to prove \eqref{ON} and assertions (1) and (2).
Let $E=\R^n$ and $\mu(\d x)$ be the Lebesgue measure. Consider the symmetric $\aa$-stable process on $\R^n$ with jumping kernel
$$J(x,y):= \1_{\{x\ne y\}} |x-y|^{-(n+\aa)},\ \ x,y\in \R^n.$$
Let $P_t$ be the Markov semigroup generated by the Dirichlet form
$$\E(f,g):= \ff  12 \int_{\R^n\times\R^n} (f(x)-f(y))(g(x)-g(y))J(x,y)\,\d x\,\d y,\ \ f,g\in  \D(\E).$$  It is well known that for some constant $c_1\ge1$, we have the heat kernel upper bound  (see for instance \cite[Theorem 3.2]{CK})
\beq\label{UP1} \|P_t\|_{L^1(\mu)\to L^\infty(\mu)} \le \ff {c_1} {t^{n/\aa}},\ \ t>0,\end{equation} as well as the gradient estimate (see for instance \cite[Theorem 1.3 and Example 1.4]{SSW})
\beq\label{UP2} \|\nn P_tf\|_\infty:=\sup_{x\in \R^n}\limsup_{y\to x}\frac{|P_tf(y)-P_tf(x)|}{|y-x|}\le \ff {c_1} {t^{1/\aa}} \|f\|_\infty,\ \ f\in L^\infty(\mu), t>0.\end{equation}
The heat kernel upper bound  \eqref{UP1} is equivalent to the Sobolev/Nash inequality with dimension $2n/\aa$ (see \cite{Davies} or \cite[Section 3]{CK}), or the super Poincar\'e inequality \eqref{SP1} with
\beq\label{BW}\bb(r)= c_2 r^{-n/\aa},\ \ r>0\end{equation} for some constant $c_2\ge1$, see \cite{W00a, W00b} or \cite{Wbook}.

Now, for any $h$ satisfying conditions (i) and (ii),  the gradient estimate \eqref{UP2} yields
\beq\label{BW2}\beg{split}\theta(t)&:= \sup_{\|g\|_\infty\le 1} {\rm ess}_{\mu\times\mu}\sup_{x\ne y} \ff{|P_t g(x)-P_t g(y)|}{h(|x-y|)}
\le \sup_{s>0} \ff {2c_1} {h(s)} \Big(1\land \ff s{t^{1/\aa}}\Big)\\
 &= 2c_1 \sup_{s>0}\Big( \ff 1 {h(s)}\land \ff{s}{h(s) t^{1/\aa}}\Big)=  \ff{2c_1}{h(t^{1/\aa})},\ \ t>0,\end{split}\end{equation}
 where the last step follows from the fact that $h(s)^{-1}$ is decreasing while $s h(s)^{-1}$ is increasing so that the sup is reached at
 $s= t^{1/\aa}$ which solves $\ff 1 {h(s)}= \ff{s}{h(s) t^{1/\aa}}.$

 Finally, let $\gg(x,y)= h(|x-y|).$ By \eqref{BW} and \eqref{BW2}, we have
 $$\Phi_\gg(s):= \int_0^s \,\d r \int_0^{\bb^{-1}(r)} \theta(t)\,\d t\le 2c_1 \int_0^s\,\d r \int_0^{(r/c_2)^{-\aa/n}} \ff{\d t}{h(t^{1/\aa})}\le c_3 \Phi_h(c_4s),\ \ s\ge 0$$ for some constants $c_3, c_4\ge 1$. Therefore, by Theorem \ref{T2.1} and the property \eqref{*WH}, we prove \eqref{ON}  for some constant $C>0.$

 Below we verify \eqref{N1}--\eqref{*22} and their sharpness respectively.

\ \newline
 (a)\textbf{ For \eqref{N1}.} Let $\aa=1$ and $h(s)= s^{1-\aa_1/2}\land s^{1-\aa_2/2}$ for $s\ge 0.$ Then
 $$\Phi_h(s):= \int_0^s \,\d t \int_0^{t^{-1/n}} \ff{\d r}{h(r)} \le c_5\big(s^{\ff{n-\aa_1/2}n}\lor s^{\ff{n-\aa_2/2}n}\big),\ \ s\ge 0$$
 holds for some constant $c_5>0$. So,
 $$N_h:= \Phi_h^{-1} \ge c_6 N^\land_{\aa_1,\aa_2}$$ holds for some constant $c_6>0$. Therefore, \eqref{N1} follows from \eqref{ON}.

To verify the sharpness of \eqref{N1}, let $N$ be a Young function  such that $N\npreceq N_{\aa_1,\aa_2}^\land$. We have
\beq\label{N11} \Big( \limsup_{s\downarrow 0}+\limsup_{s\uparrow \infty}\Big)  \frac{N(s) }{N_{\aa_1,\aa_2}^\land(s)}=\infty.\end{equation}  Let
 $$f_{s}(x):= (s-|x|)^+,\ \ s>0, x\in\R^n.$$ Then
$$|f_s(x)-f_s(y)|\le (s\land |x-y|)\big\{\1_{B(0,s)}(x) +\1_{B(0,s)}(y)\big\}.$$
Thus,  there exist constants $c_7,c_8>0$ such that
\beq\label{N13} \beg{split}&\int_{\R^n\times\R^n} \ff{|f_{s}(x)-f_{s}(y)| }{|x-y|^{n+\aa_1/2}\lor |x-y|^{n+\aa_2/2}} \,\d x\,\d y\\
&\le 2 \int_{B(0,s)}\d x\int_{\R^n} \ff{s\land |x-y| }{|x-y|^{n+\aa_1/2}\lor |x-y|^{n+\aa_2/2}}  \,\d y\\
&\le c_7 \int_{B(0,s)} (s^{1-\aa_1/2}\land s^{1-\aa_2/2})\,\d x\le
  c_8 (s^{n+1-\aa_1/2}\land s^{n+1-\aa_2/2}),\ \ s\ge 0.\end{split}\end{equation} If \eqref{N1} holds for $N$ replacing $N_{\aa_1,\aa_2}^\land$, then
$$ \|f_{s}\|_N\le c_0 (s^{n+1-\aa_1/2}\land s^{n+1-\aa_2/2}),\ \ s>0$$ holds for some constant $c_0>0.$  Therefore, there exist constants $c_9,c_{10}>0$ such that
\beq\label{NM}\beg{split} 1&\ge \int_{\R^n} N\Big(\big\{c_0 (s^{n+1-\aa_1/2}\land s^{n+1-\aa_2/2})\big\}^{-1}f_{s}(x)\Big)\, \d x\\
&\ge \int_{\{s/4\le |x|\le 3s/4\}} N\big(c_9 (s^{\aa_1/2-n}\lor s^{\aa_2/2-n})\big)\,\d x\\
&\ge c_{10} s^n N\big(c_9 (s^{\aa_1/2-n}\lor s^{\aa_2/2-n})\big),\ \ s>0.\end{split}\end{equation}
Noting that
$$\inf_{s>0}  s^n N_{\aa_1,\aa_2}^\land \big(c_9(s^{\aa_1/2-n}\lor s^{\aa_2/2-n})\big)>0,$$
from  \eqref{N11} and \eqref{NM} we conclude   that
$$1\ge \Big(\limsup_{s\downarrow 0} +\limsup_{s\uparrow \infty}\Big) c_{10} s^n N\big(c_9 (s^{\aa_1/2-n}\lor s^{\aa_2/2-n})\big)=\infty,$$ which is impossible.

\ \newline
(b) \textbf{For \eqref{N2}.}  Let $\aa=1$ and $h(s)= s^{1-\aa_1/2}\lor s^{1-\aa_2/2}$ for $s\ge 0.$ Then
 $$\Phi_h(s):= \int_0^s\, \d t \int_0^{t^{-1/n}} \ff{\d r}{h(r)} \le c_5'\big(s^{\ff{n-\aa_1/2}n}\land s^{\ff{n-\aa_2/2}n}\big),\ \ s\ge 0$$
 holds for some constant $c_5'>0$. So,
 $$N_h:= \Phi_h^{-1} \ge c_6' N^\lor_{\aa_1,\aa_2}$$ holds for some constant $c_6'>0$. Therefore, \eqref{N2} follows from \eqref{ON}.

As in (a), we have
$$\int_{\R^n\times\R^n} \ff{|f_{s}(x)-f_{s}(y)| }{|x-y|^{n+\aa_1/2}\land |x-y|^{n+\aa_2/2}} \d x\d y
\le c_7' (s^{n+1-\aa_1/2}\lor s^{n+1-\aa_2/2}),\ \ s>0.$$ for some constant $c_7'>0$. Moreover,
for any Young function $N$ with $N\npreceq N_{\aa_1,\aa_2}^\land$, and for any constant $c_0>0$, we have
$$\Big(\limsup_{s\downarrow 0} +  \limsup_{s\uparrow\infty} \Big) \int_{\R^n} N\Big(\big\{c_0 (s^{n+1-\aa_1/2}\lor s^{n+1-\aa_2/2})\big\}^{-1}f_{s}(x)\Big)  \d x=\infty,$$
so that \eqref{N2} does not hold for $N$ replacing $N_{\aa_1,\aa_2}^\lor$.

\ \newline
(c) \textbf{For \eqref{*21}.} Let $\lambda_0\ge 2$ such that
$$h(s):= s^{\aa/2} \{\log(\ll_0+s^{-1})\}^q,\ \ s\ge 0$$ satisfies condition (i).  Then there exists a constant $c_{11}>0$ such that
$$\Phi_h(s):= \int_0^s \d t \int_0^{t^{-\ff 1 n}} \ff{r^{(\aa/2)-1}}{\{\log(\ll_0+r^{-1})\}^q}\,\d r \le \ff{c_{11} s^{\ff{n-\aa/2}n}}{\{\log(2+s)\}^q},\ \ s\ge 0.$$ Thus,
$$N_h:= \Phi_h^{-1} \ge c_{12} N_\aa^{log,q,+}$$ holds for some constant $c_{12}>0$. Therefore, \eqref{*21} follows from \eqref{ON}. The sharpness can be verified with reference functions $f_s$ as above.

\ \newline
(d) \textbf{For \eqref{*22}.} We take
$$h(s):= s^{\aa/2} \{\log(\ll_0+s)\}^p,\ \ s\ge 0$$ for some $\ll_0\ge 2$ large enough such that condition (i) is satisfied. Then there is a constant $c'_{11}>0$ such that
$$\Phi_h(s):= \int_0^s \d t \int_0^{t^{-\ff 1 n}} \ff{r^{(\aa/2)-1}}{\{\log(\ll_0+r)\}^p}\,\d r \le \ff{c'_{11} s^{\ff{n-\aa/2}n}}{\{\log(2+s^{-1})\}^p},\ \ s\ge 0.$$ Hence,
$$N_h:= \Phi_h^{-1} \ge c'_{12} N_\aa^{log,p,-}$$ holds for some constant $c'_{12}>0$. Therefore, from \eqref{ON} we can get \eqref{*22}. Similar to (c), one can verify the sharpness of \eqref{*22} by using reference functions $f_s$ as above.

\subsection{Extension to the truncated $\aa$-stable form}

\beg{thm}\label{T1.2} Let $n\ge 2$, $\aa\in (0,2)$, and let $h: (0,\infty)\to (0,\infty)$ satisfy  condition {\rm (i)} in Theorem $\ref{T1.1}$ and
\beg{enumerate}\item[{\rm (ii')}] $$\tt \Phi_h(s):= \int_0^s \,\d t \int_0^{t^{-\ff\aa n} \lor t^{-\ff 2 n}} \ff{\d r}{h(r^{\ff 1 \aa}\land r^{\ff 1 2})} <\infty,\ \ s\ge 0.$$\end{enumerate}
Let $\tt N_h= \tt \Phi_h^{-1}$. Then there exists a constant $C>0$ such that
\beq\label{ON2} \|f\|_{\tt N_h} \le C \int_{{\{|x-y|\le 1\}}} |f(x)-f(y)|\ff{ h(|x-y|)}{\,\,|x-y|^{n+\aa}} \,\d x \,\d y,\ \ f\in L_{\tt N_h}(\R^n).\end{equation}
Consequently, for $\tt N_{\aa}(s):= s^{\ff{n}{n-\aa/2}}\land s^{\ff n {n-1}},$ there exists a constant $C>0$ such that
\beq\label{N3} \|f\|_{\tt N_\aa} \le C  \int_{{\{|x-y|\le 1\}}} \ff{ |f(x)-f(y)| }{|x-y|^{n+\aa/2}}\, \d x \,\d y,\ \ f\in L_{\tt N_\aa}(\R^n).\end{equation}
This inequality fails if $\tt N_\aa$ is replaced by a Young function    $N\npreceq\tt N_\aa$.
\end{thm}

\beg{proof} Consider the following truncated $\aa$-stable Dirichlet form
$$\E(f,g):= \ff  12 \int_{\{|x-y|\le1\}} (f(x)-f(y))(g(x)-g(y)) \frac{1}{|x-y|^{d+\aa}}\,\d x\,\d y,\ \ f,g\in  \D(\E).$$
Let $P_t$ be the associated Markov semigroup. Then, by \cite[Proposition 2.2]{CKK} and \cite[Theorem 1.3 and Example 1.5]{SSW}, we have
\beq\label{UPt1} \|P_t\|_{L^1(\mu)\to L^\infty(\mu)} \le \ff {c_1} {t^{n/\aa}\wedge t^{n/2}},\ \ t>0,\end{equation} and
\beq\label{UPt2} \|\nn P_tf\|_\infty\le \ff {c_1} {t^{1/\aa}\wedge t^{1/2}} \|f\|_\infty,\ \ f\in L^\infty(\mu), t>0\end{equation} for some constant $c_1\ge1$.
By \eqref{UPt1} and \cite[Theorem 4.5]{W00b}, the super Poincar\'e inequality \eqref{SP1} holds with
\beq\label{BWt}\bb(r)= c_2 (r^{-n/\aa}\vee r^{-n/2}),\ \ r>0\end{equation} for some constant $c_2\ge1$.
On the other hand, by \eqref{UPt2} and the argument of \eqref{BW2},
for any $h$ satisfying condition (i),
\beq\label{BWt2}\beg{split}\theta(t):&= \sup_{\|g\|_\infty\le 1} {\rm ess}_{\mu\times\mu}\sup_{x\ne y} \ff{|P_t g(x)-P_t g(y)|}{h(|x-y|)}\\
&\le \sup_{s>0} \ff {2c_1} {h(s)} \Big(1\land \ff s{t^{1/\aa}\wedge t^{1/2}}\Big)=\ff{2c_1}{h(t^{1/\aa}\wedge t^{1/2})},\ \ t>0.\end{split}\end{equation}
Thus, let $\gg(x,y)= h(|x-y|).$ By \eqref{BWt} and \eqref{BWt2}, for $s>0$,
 $$\Phi_\gg(s):= \int_0^s \,\d r \int_0^{\bb^{-1}(r)} \theta(t)\,\d t\le 2c_1 \int_0^s\,\d r \int_0^{(r/c_2)^{-\aa/n}\vee (r/c_2)^{-2/n}} \ff{\d t}{h(t^{1/\aa}\wedge t^{1/2})}\le c_3 \tilde \Phi_h(c_4s)$$ holds with some constants $c_3, c_4\ge 1$. Therefore, by (ii') and Theorem \ref{T2.1} we prove \eqref{ON2}  for some constant $C>0.$

 Below we verify \eqref{N3} and its sharpness. Let $h(s)= s^{\aa/2}\lor s$ for $s\ge 0.$ Then
 $$\tilde\Phi_h(s):=\int_0^s\,\d r \int_0^{r^{-\aa/n}\vee r^{-2/n}} \ff{\d t}{h(t^{1/\aa}\wedge t^{1/2})} \le c_5\big(s^{\ff{n-\aa/2}n}\lor s^{\ff{n-1}n}\big),\ \ s\ge 0$$
 holds for some constant $c_5>0$, where in the inequality we have used the fact that $n\ge2$. So,
 $$\tilde N_h:= \tilde \Phi_h^{-1} \ge c_6 \tilde N_{\aa}$$ holds for some constant $c_6>0$. Therefore, \eqref{N3} follows from \eqref{ON}.

Let
$f_s$ be the function in the argument of Theorem \ref{T1.1}. Then there exists a constant $c_7>0$ such that
\begin{equation*} \beg{split}\int_{\{|x-y|\le 1\}} \ff{|f_{s}(x)-f_{s}(y)| }{|x-y|^{n+\aa}} \,\d x\,\d y
&\le 2 \int_{B(0,s)}\d x\int_{\{|y-x|\le 1\}} \ff{s\land |x-y| }{|x-y|^{n+\aa/2}}  \,\d y\\
& \le  c_7 (s^{n+1-\aa/2}\land s^{n}),\ \ s\ge 0.\end{split}\end{equation*} Let $N$ be a Young function  such that $N\npreceq \tilde N_{\aa}$. We have
$$ \Big( \limsup_{s\downarrow 0}+\limsup_{s\uparrow \infty}\Big)  \frac{N(s) }{\tilde N_{\aa}(s)}=\infty.$$ Suppose that \eqref{N3} holds for $N$. Then
$$ \|f_{s}\|_N\le c_0 (s^{n+1-\aa/2}\land s^{n}),\ \ s>0$$ holds for some constant $c_0>0,$  so that
\begin{equation*}\beg{split} 1&\ge \int_{\R^n} N\Big(\big\{c_0 (s^{n+1-\aa/2}\land s^{n})\big\}^{-1}f_{s}(x)\Big)\, \d x\\
&\ge \int_{\{s/4\le |x|\le 3s/4\}} N\big(c_8 (s^{\aa/2-n}\lor s^{1-n})\big)\,\d x\ge c_{9} s^n N\big(c_8 (s^{\aa/2-n}\lor s^{1-n})\big),\ \ s>0.\end{split}\end{equation*}
Combining this with all the estimates above, we obtain that
$$1\ge \Big(\limsup_{s\downarrow 0} +\limsup_{s\uparrow \infty}\Big) c_{9} s^n N\big(c_8 (s^{\aa/2-n}\lor s^{1-n})\big)=\infty,$$ which is impossible. Therefore, \eqref{N3} does not hold for $N$, and so we verify the sharpness of \eqref{N3}.
\end{proof}

\subsection{Extension to discrete $\alpha$-stable Dirichlet form}
In this subsection, let $E=\Z^n$ and $\mu$ be the counting measure. Under this setting, the Orlicz norm $\|\cdot\|_N$ for a Young function $N$ is essentially determined by $N(s)$ for small $s>0$. In particular,  for any two Young functions $N_1$ and $N_2$, $\|\cdot\|_{N_1} \le c \|\cdot\|_{N_2}$ for some constant $c>0$ if and only if there is a constant $c'>0$ such that $N_1(s)\le c'N_2(s)$ holds for all $s\in [0,1].$
Moreover, since $|x-y|\ge 1$ for $x\ne y$,  we have
$2\le 2+|x-y|^{-1} \le 3$, and for $\aa_1\le \aa_2$,
$$|x-y|^{n+\aa_1/2} \lor |x-y|^{n+\aa_2/2}=|x-y|^{n+\aa_2/2},\ \ |x-y|^{n+\aa_1/2} \land |x-y|^{n+\aa_2/2}=|x-y|^{n+\aa_1/2}.$$ Therefore, in assertion (1) of Theorem $\ref{T1.1}$ we will take $\aa_1=\aa_2=\aa$, and in assertion (2) we only consider $N_\aa^{\log, p,-}:=  \big\{s\log^p\left(\ll+ s^{-1}\right)\big\}^{\ff n{n-\aa/2}}$ with some constant $\ll\ge2$.

\begin{thm} Assertions in Theorem $\ref{T1.1}$ hold for the counting measure $\mu$ on $\mathbb Z^n$ replacing the Lebesgue measure on $\R^n$. More precisely, for any $\aa\in (0,2)$ and $h\in \H_\aa$, there exists a constant $C>0$ such that
$$ \|f\|_{N_h} \le C \sum_{x,y\in\Z^n,x\ne y}  |f(x)-f(y)|\ff{ h(|x-y|)}{\,\,|x-y|^{n+\aa}},\ \ f\in L_{N_h}(\Z^n).$$
Consequently: \beg{enumerate} \item[$(1)$] There exists a constant $C>0$ such that
$$ \|f\|_{\ff n{n-\aa/2}}  \le C \sum_{x,y\in\Z^n,x\ne y} \ff{|f(x)-f(y)| }{|x-y|^{n+\aa/2} },\ \ f\in L^{\ff n {n-\aa/2}}(\Z^n).$$
This inequality fails if $\|\cdot\|_{\ff n{n-\aa/2}}$ is replaced by $\|\cdot\|_{L_N}$ for a Young function $N$ with $$\limsup_{s\to 0} s^{\ff {n-\aa/2}{n}} N(s) =\infty.$$
\item[$(2)$] For any $\aa\in (0,2)$ and $p\in \R$,   there exists a constant $C>0$ such that for all $f\in L_{N^{log,p,-}_{\aa}}(\Z^n),$
$$ \|f\|_{N_\aa^{log,p,-}}  \le C \sum_{x,y\in\Z^n,x\ne y}|f(x)-f(y)| \ff{ \{\log(2+|x-y|)\}^p}{|x-y|^{n+\aa/2}}.$$
This inequality  fails  if $N_\aa^{log,p,-}$  is replaced by a Young function $N$ with $$\limsup_{s\to 0} \ff{N(s)}{N_\aa^{log,p,-}(s)} =\infty.$$    \end{enumerate}
\end{thm}

\begin{proof} According to the proof of Theorem \ref{T1.1}, it suffices to construct a symmetric sub-Markov semigroup $P_t$ on $L^2(\mu)$ such that the associated   Dirichlet form $(\E,\D(\E))$ is comparable with
$$\E_\aa(f,f):= \ff 1 2 \sum_{x,y\in \Z^n, x\ne y} \ff{(f(x)-f(y))^2 }{|x-y|^{n+\aa}},\ \ \D(\E_\aa):= \{f\in L^2(\mu), \E_\aa(f,f)<\infty\},$$
i.e., $\D(\E)= \D(\E_\aa)$ and there exists a constant $C\ge 1$ such that
\beq\label{CP} C^{-1}  \E(f,f)\le \E_\aa(f,f)\le C \E(f,f),\ \ f\in \D(\E)=\D(\E_\aa);\end{equation} and moreover, both \eqref{UP1} and \eqref{UP2} are satisfied for $P_t$, where in \eqref{UP2}  $$\|\nn P_tf\|_\infty := \sup_{x,y\in \Z^n, |x-y|=1} |P_t f(x)-P_t f(y)|.$$

Condition \eqref{UP1} can be easily deduced from the Nash inequality for $(\E_\aa,\D(\E_\aa))$ (see for example \cite[Proposition 2.1]{MS2}), but to prove 
explicit gradient estimate 
\eqref{UP2} we need additional arguments. Below we first construct the required semigroup $P_t$ then verify these two estimates.

 (1) \textbf{Construction of $P_t$.}  Let $q_k(x,y)$ and $Q^k$ be the transition function and the semigroup for discrete time simple random walk $Y'=(Y'_k)_{k\ge0}$ on $\Z^n$, respectively. It is known (see \cite{Te} or \cite[Theorem 5.1]{HS}) that there are constants $c_i>0 $ $(i=1,\cdots 5)$ so that
\begin{equation}\label{e:3.15}q_k(x,y)\le \frac{c_1}{k^{n/2}}\exp\left(\frac{-c_2|x-y|^2}{k}\right),\quad x,y\in \Z^n, k\ge 1,\end{equation}
\begin{equation}\label{e:3.16}q_k(x,y)+q_{k+1}(x,y)\ge \frac{c_3}{k^{n/2}}\exp\left(\frac{-c_4|x-y|^2}{k}\right),\quad x,y\in \Z^n, k\ge |x-y|\end{equation} and
\begin{equation}\label{e:3.17}\begin{split}\|\nabla Q^kf\|_\infty:=&\sup_{x\in \Z^n}\sup_{y\in \Z^n,|y-x|=1}|Q^kf(y)-Q^kf(x)| \\
\le & c_5 k^{-1/2} \|f\|_\infty,\quad k\ge 1, f\in L^\infty(\Z^n).\end{split}\end{equation}

Consider the discrete subordination of $Y'$ by the Bernstein function $\psi(r)=r^{\aa/2}$ with $\aa\in (0,2)$, see \cite{BS}. Denote by $X'=(X'_k)_{k\ge0}$ the corresponding discrete time subordinated Markov chain on $\Z^n$, and by $p_k(x,y)$ the transition function of $X'$. Then, according to \cite[Proposition 2.3 and Example 2.1]{BS}, \begin{equation}\label{e:3.18}p_1(x,y)=\sum_{k=1}^\infty c(\psi,k)q_k(x,y),\ \ k\ge 1, x,y\in \Z^n,\end{equation} where
$$c(\psi,k)= \frac{\aa}{2\Gamma(1-\aa/2)}\frac{\Gamma(k-\aa/2)}{\Gamma(k+1)}.$$
We  claim that \begin{equation}\label{e:3.19}\frac{c_0^{-1}}{|x-y|^{n+\aa}}\le p_1(x,y)\le \frac{c_0}{|x-y|^{n+\aa}},\quad x,y\in \Z^n, x\neq y\end{equation} holds for some constant $c_0\ge1$. Indeed, by \cite{TE},
\begin{equation}\label{e:3.20}c_6^{-1}k^{-(\aa/2+1)}\le c(\psi,k)\le c_6k^{-(\alpha/2+1)},\ \ k\ge 1\end{equation} holds for some constant $c_7\ge1$. Then, by \eqref{e:3.16}, \eqref{e:3.18} and \eqref{e:3.20}, we have
\begin{align*} p_1(x,y)&\ge \frac{1}{2}\bigg(\sum_{k=1}^\infty c(\psi,k)q_k(x,y)+\sum_{k=1}^\infty c(\psi,k+1)q_{k+1}(x,y)\bigg)\\
&\ge c_7\sum_{k=1}^\infty \frac{1}{k^{1+\aa/2}}(q_k(x,y)+q_{k+1}(x,y))\ge c_8\sum_{k=|x-y|^2}^\infty \frac{1}{k^{1+(\aa+n)/2}}\ge \ff{c_9}{|x-y|^{n+\aa}}.\end{align*} On the other hand, according to \eqref{e:3.15}, \eqref{e:3.18} and \eqref{e:3.20},
\begin{align*} p_1(x,y)&\le  c_7\sum_{k=1}^\infty \frac{1}{k^{1+\aa/2}}q_k(x,y)\\
&\le  c_9\left(\sum_{k=|x-y|^2}^\infty \frac{1}{k^{1+(\aa+n)/2}}+ \sum_{k=1}^{|x-y|^2}  \frac{1}{k^{1+(\aa+n)/2}}\exp\left(\frac{-c_2|x-y|^2}{k}\right)\right)\\
&\le \ff{c_{10}} {|x-y|^{n+\aa}}.\end{align*} Thus, \eqref{e:3.19} is proved.

Let $N_t$ be a Poisson process independent of $X'$ and $Y'$. Set
$X_t=X'_{N_t}$ and $Y_t=Y'_{N_t}$ for all $t\ge0$. Therefore,  by \eqref{e:3.19},  $X=(X_t)_{t\ge0}$ is a continuous time symmetric Markov chain on $\Z^n$ such that the associated Dirichlet form $(\E,\D(\E))$ is comparable with $\E_\aa$, i.e., \eqref{CP} holds for some constant $C>1.$
Let $P_t$ be the Markov semigroup of $X_t$.

(2) \textbf{Proofs of \eqref{UP1} and \eqref{UP2}.}  Let $Q_t$ be the Markov semigroup of $Y_t$. We have
$$Q_t=\e^{-t}\sum_{k=0}^\infty \frac{ t^k Q^k}{k!},\quad t>0.$$
Then, by \eqref{e:3.15}, for any $f\in L^\infty(\Z^n)$ and $t>0$,
\begin{equation}\label{e:3.21}\begin{split}\|Q_tf\|_\infty&\le \e^{-t}\|f\|_\infty+ c_1\e^{-t} \sum_{k=1}^\infty \frac{ t^k k^{-n/2}}{k!}\|f\|_\infty\\
&\le \e^{-t}\|f\|_\infty+ c_{11}t^{-d/2}\|f\|_\infty\le  \frac{c_{12} }{t^{n/2}}\|f\|_\infty,\end{split}\end{equation} where in the second inequality we have used the expansion for inverse moments of Poisson distribution, see \cite[(29) in Corollary 3]{Z}. By \eqref{e:3.17}, we also have that for any $f\in L^\infty(\Z^n)$ and $t>0$,
\begin{equation}\label{e:3.22}\begin{split}\|\nabla Q_tf\|_\infty&\le \e^{-t}\|f\|_\infty+ c_6\e^{-t} \sum_{k=1}^\infty \frac{ t^k k^{-1/2}}{k!}\|f\|_\infty\le \frac{c_{13} }{t^{1/2}}\|f\|_\infty,\end{split}\end{equation} where in the last inequality we have used again
 \cite[(29) in Corollary 3]{Z}.

On the other hand, let $S_t$ be the $\aa/2$-subordinator, which is independent of $X,X',Y$ and $Y'$. According to \cite[Proposition 1.2]{An}, we know that $X_t$ and $Y_{S_t}$ enjoy the same distribution. That is,
$$P_tf=\int_0^\infty Q_sf\,\P(S_t\in \d s),\qquad t\ge 0, f\in L^\infty(\Z^n).$$
This along with \eqref{e:3.21} and \eqref{e:3.22} yields
$$\frac{\|P_tf\|_\infty}{\|f\|_\infty}\le c_{12}\int_0^\infty s^{-n/2}\, \P(S_t\in \d s)\le c_{14} t^{-n/\alpha},\ \ t\ge 0, f\in L^\infty(\Z^n)$$  and
$$\frac{\|\nabla P_tf\|_\infty}{\|f\|_\infty}\le c_{13}\int_0^\infty s^{-1/2}\, \P(S_t\in \d s)\le c_{15} t^{-1/\alpha},\ \ t\ge 0, f\in L^\infty(\Z^n),$$ where we used the fact that $\mathbb{E} S_t^{-\ll}= c_{\aa,\ll} t^{-2\ll/\aa}$ for all $t,\ll>0$, see \cite[(25.5)]{Sato}.  Therefore,    \eqref{UP1} and \eqref{UP2} hold.
 \end{proof}

\section{Isoperimetric inequalities for $\aa$-stable-like Dirichlet forms: a perturbation argument }
Let $n\ge 2$ and $\aa\in (0,2)$.
Let $W\in \B(\R^n)$ be such that $\mu_W(\d x):= \e^{-W(x)}\,\d x$ is a probability measure. Consider the following $\aa$-stable-like Dirichlet form $(\E_{\aa,W},\D(\E_{\aa,W}))$:
\begin{equation}\label{e:stable}\begin{split}\E_{\aa,W}(f,f):=& \int_{\R^n\times \R^n} \ff{|f(x)-f(y)|^2}{|x-y|^{n+\alpha}}\,\d y\,\mu_W(\d x), \\
\D(\E_{\aa,W}):=&\{f\in L^2(\R^n,\mu_W): \D(\E_{\aa,W})(f,f)<\infty\}.\end{split} \end{equation} Obviously,    $C_c^\infty(\R^n)\subset\D(\E_{\aa,W}).$  See \cite{WW, CWW} for explicit criteria of Poincar\'e-type (i.e.,  Poincar\'e, weak Poincar\'e and super Poincar\'e) inequalities for this Dirichlet form.

  Since it is not clear how to verify the regularity property (e.g.\ gradient estimates) for the associated semigroup $P_t$,  we could not apply Theorem \ref{T2.1} for non-negative symmetric function $\gg(x,y)$ satisfying  $\gg(x,y)\to 0$ as $y\to x$.  So, in this section, we will establish isoperimetric inequalities for $(\E_{\aa,W},\D(\E_{\aa,W}))$  by using a perturbation argument. The main result of this section is the following.

\begin{thm}\label{T:thm} Let $n\ge 2$ and $\aa\in (0,2)$. Let $W\in \B(\R^n)$ be such that $\mu_W(\d x)= \e^{-W(x)}\,\d x$ is a probability measure. Set
$$\Phi(l):= \inf_{|x|\ge l} \frac{\e^{W(x)}}{|x|^{n+\aa/2}},\quad l\ge 1.$$
\begin{itemize}
\item[$(1)$] If $\lim_{l\to\infty}\Phi(l) >0,$ then there are constants $c_1,c_2>0$ such that for any $f\in C_c^\infty(\R^n)$,  \begin{equation}\label{e:the1}\mu_W(f^2)\le c_1\int_{\R^n\times \R^n} \ff{|f^2(x)-f^2(y)|}{|x-y|^{n+\alpha/2}}\,\d y\,\mu_W(\d x)+c_2 \mu_W(|f|)^2.\end{equation}
\item[$(2)$] Let $W$ be locally Lipschitz continuous. If $\lim_{l\to\infty} \Phi(l)=\infty,$
then for any $r>0$ and $f\in C_c^\infty(\R^d)$,
\begin{equation}\label{e:the2}\mu_W(f^2)\le r\int_{\R^n\times \R^n} \ff{|f^2(x)-f^2(y)|}{|x-y|^{n+\alpha/2}}\,\d y\,\mu_W(\d x)+\bb(r) \mu_W(|f|)^2,\end{equation} where
\begin{align*}\beta(r)=\inf\bigg\{&2c_1\left(s^{-2n/\aa}+s^{-n}\right)\left(\Hup_{|z|\le l+1}\e^{W(z)/2}\right):s+\frac{1}{\Phi(l-1)} \le c_2(r\wedge1)\\
&  \qquad\qquad\qquad\qquad\qquad \qquad \qquad \qquad \text{ and } \Hup_{|z|\le l+2} \e^{2|\nabla W(z)|}\le \ff {c_3}{s}  \bigg\}\end{align*} with some constants $c_1,c_2,c_3>0$.
\end{itemize}\end{thm}

To prove Theorem \ref{T:thm}, we will make perturbation to   the following Poincar\'e type isoperimetric inequality for the truncated $\aa$-stable Dirichlet from on $\R^n.$

\begin{lem}\label{L5.2} There is a constant $c>0$ such that for all $f\in C_c^\infty(\R^n)$,
$$\int f(x)^2\,\d x\le r \int_{\{|x-y|\le 1\}} \frac{|f^2(x)-f^2(y)|}{|x-y|^{n+\aa/2}}\,\d x\,\d y+ c(r^{-2n/\aa}+r^{-n})\left(\int|f|(x)\,\d x\right)^2,\ \ r>0.$$ \end{lem}
\begin{proof}This follows from \eqref{N3} and Corollary \ref{C4.1}(1).\end{proof}

For any $D\subset \R^n$, consider the isoperimetric constant
\begin{equation}\label{e:ff001}\begin{split}\kk_W(D):=&\frac{1}{2} \inf\left\{ \frac{1 }{\mu_W(A)}\int_{A\times  A^c}\frac{\e^{W(x)}+\e^{W(y)}}{|x-y|^{n+\aa/2}}\,\mu_W(\d x)\,\mu_W(\d y):A\Subset D, \mu_W(A)>0\right\}\\
=&\inf\left\{\frac{1}{\mu_W(f)}\int\frac{|f(x)-f(y)|}{|x-y|^{n+\aa/2}}\,\d y\,\mu_W(\d x): f\ge0, f|_{D^c}=0, \mu_W(f)>0 \right\},\end{split}\end{equation} where the second equality in \eqref{e:ff001} can be verified by the co-area formula, see \cite[Theorem 3.1]{LS}.

\begin{lem}\label{L:1}Let $n\ge2$ and $\aa\in (0,2)$.   Let
$B_l=\{|\cdot|<l\}$ for $l>0.$
\begin{itemize}
\item[$(1)$] If \begin{equation}\label{e:ps11}\lim_{l\to\infty} \kappa_W(B_{l}^c)>0,\end{equation} then \eqref{e:the1} holds with some constants $c_1,c_2>0$.
\item[$(2)$] Let $W$ be locally Lipschitz continuous. If \begin{equation}\label{e:ps12}\lim_{l\to\infty} \kappa_W(B_{l}^c)=\infty,\end{equation} then \eqref{e:the2} holds with \begin{align*}\beta(r)=\inf\bigg\{&2c_1\left(s^{-2n/\aa}+s^{-n}\right)\left(\Hup_{|z|\le l+1}\e^{W(z)/2}\right):s+\frac{1}{\kappa_W(B_{l-1}^c)} \le c_2(r\wedge1)\\
& \qquad\qquad\qquad\qquad\qquad \qquad \qquad \qquad  \text{ and } \Hup_{|z|\le l+2} \e^{2|\nabla W(z)|}  \le \frac{c_3}{s} \bigg\}\end{align*} for some constants $c_1,c_2,c_3>0$.
\end{itemize}
\end{lem}

\begin{proof}For any $l> k\ge 1$, let $\psi_{l,k}\in C_c^1(\R^n)$ such that $\psi_{l,k}(x)=1$ for all $|x|\le l$, $\psi_{l,k}(x)=0$ for all $|x|>l+k$, and $|\nabla \psi_{l,k}|\le 2/k$ on $\R^n$. Then, according to Lemma \ref{L5.2}, for any $f\in C_c^\infty(\R^n)$ and $r>0$,
\begin{align*}&\int f^2(x)\e^{-W(x)}{\bf1}_{\{|x|\le l\}}\,dx\\
&\le \int f^2(x)\e^{-W(x)}\psi_{l,k}(x)\,\d x\\
&\le r \int_{\{|x-y|\le 1\}} \frac{ |f^2(x)\e^{-W(x)}\psi_{l,k}(x)-f^2(y)\e^{-W(y)}\psi_{l,k}(y)| }{|x-y|^{n+\aa/2}}\,\d x\, \d y\\
&\quad +c_1\left(r^{-2n/\aa}+r^{-n}\right) \left(\int |f|(x) \e^{-W(x)/2} \psi_{l,k}(x)^{1/2}\,\d x\right)^2\\
&\le r \int_{\{|x-y|\le 1\}} \frac{ |f^2(x)-f^2(y)| }{|x-y|^{n+\aa/2}}\, \d y\,\e^{-W(x)}\psi_{l,k}(x)\,\d x\\
&\quad +r \int \,f^2(x)\e^{-W(x)}\,\d x \int_{\{|x-y|\le 1\}} \frac{ |\psi_{l,k}(x)-\psi_{l,k}(y)| }{|x-y|^{n+\aa/2}}\, \d y\\
&\quad+ r \int \,f^2(x)\e^{-W(x)}\,\d x \int_{\{|x-y|\le 1\}} \frac{ |1-\e^{W(x)-W(y)}| }{|x-y|^{n+\aa/2}}\, \psi_{l,k}(y)\, \d y\\
&\quad+c_1\left(r^{-2n/\aa}+r^{-n}\right)\left(\Hup_{|z|\le l+k}\e^{W(z)/2}\right) \left(\int |f|(x) \e^{-W(x)} \,\d x\right)^2\\
&\le r \int_{\{|x-y|\le 1\}} \frac{ |f^2(x)-f^2(y)| }{|x-y|^{n+\aa/2}}\, \d y\,\e^{-W(x)}\,\d x\\
&\quad +\frac{c_2r}{k} \left(\int \,f^2(x)\e^{-W(x)}\,\d x\right) + c_3r \left(\Hup_{|z|\le l+k+1} \e^{2|\nabla W(z)|}\right)\left(\int \,f^2(x)\e^{-W(x)}\d x\right) \\
&\quad+c_1\left(r^{-2n/\aa}+r^{-n}\right)\left(\Hup_{|z|\le l+k}\e^{W(z)/2}\right) \left(\int |f|(x) \e^{-W(x)} \,\d x\right)^2, \end{align*}
where in the last inequality we have used the facts that  for all $x\in \R^d$,
\begin{equation}\label{ee-00}\int_{\{|x-y|\le 1\}} \frac{ |\psi_{l,k}(x)-\psi_{l,k}(y)| }{|x-y|^{n+\aa/2}}\, \d y\le  \big(\Hup_{z\in \R^n} |\nabla\psi_{l,k}|(z)\big)\int_{\{|x-y|\le 1\}} \frac{1 }{|x-y|^{n-1+\aa/2}}\, \d y\le \frac{c_2}{k}\end{equation} and, by the elementary inequality $|\e^r-1|\le \e^{|r|}|r|\le \e^{2|r|}$ for all $r\in \R$,
\begin{align*}&\int_{\{|x-y|\le 1\}} \frac{ |1-\e^{W(x)-W(y)}| }{|x-y|^{n+\aa/2}} \psi_{l,k}(y) \,\d y\\
&\le \Hup_{|x|\le l+k+1}\int_{\{|x-y|\le 1, |y|\le l+k\}} \frac{ |1-\e^{W(x)-W(y)}| }{|x-y|^{n+\aa/2}}\, \d y \\
&\le  \Hup_{|x|\le l+k+1}\int_{\{|x-y|\le 1, |y|\le l+k\}} \frac{ \e^{|W(x)-W(y)|}|W(x)-W(y)| }{|x-y|^{n+\aa/2}}\, \d y \\
&\le \Big(\Hup_{|z|\le l+k+1} \e^{|\nabla W(z)|}|\nabla W(z)|\Big)\int_{\{|x-y|\le 1\}} \frac{1 }{|x-y|^{n-1+\aa/2}}\, \d y\\
&\le c_3 \Big(\Hup_{|z|\le l+k+1} \e^{2|\nabla W(z)|}\Big).\end{align*}

On the other hand, by the definition of $\kappa_W(B_{l-k}^c)$, we have \begin{align*}&\int f^2(x)\e^{-W(x)}\left(1-{\bf1}_{\{|x|\le l\}}\right)\,dx\\
&\le \int f^2(x)\e^{-W(x)}(1-\psi_{l-k,k}(x))\,\d x\\
&\le \frac{1}{\kappa_W(B_{l-k}^c)} \int\frac{ |f^2(x)(1-\psi_{l-k,k}(x))-  f^2(y)(1-\psi_{l-k,k}(y))| }{|x-y|^{n+\aa/2}}\,\d y \,\e^{-W(x)}\,\d x\\
&\le  \frac{1}{\kappa_W(B_{l-k}^c)} \int\frac{ |f^2(x)-  f^2(y)| }{|x-y|^{n+\aa/2}}\,\d y \,\e^{-W(x)}\,\d x\\
&\quad + \frac{1}{\kappa_W(B_{l-k}^c)} \int\frac{ |f^2(x)\psi_{l-k,k}(x)-  f^2(y)\psi_{l-k,k}(y)| }{|x-y|^{n+\aa/2}}\,\d y \,\e^{-W(x)}\,\d x\\
&\le  \frac{2}{\kappa_W(B_{l-k}^c)} \int\frac{ |f^2(x)-  f^2(y)| }{|x-y|^{n+\aa/2}}\,\d y\, \e^{-W(x)}\,\d x\\
&\quad + \frac{1}{\kappa_W(B_{l-k}^c)} \int f^2(x)\e^{-W(x)}\,\d x \int\frac{ |\psi_{l-k,k}(x)-  \psi_{l-k,k}(y)| }{|x-y|^{n+\aa/2}}\,\d y\\
&\le  \frac{2}{\kappa_W(B_{l-k}^c)} \int\frac{ |f^2(x)-  f^2(y)| }{|x-y|^{n+\aa/2}}\,\d y\, \e^{-W(x)}\,\d x\\
&\quad + \frac{c_2}{k\,\kappa_W(B_{l-k}^c)} \int f^2(x)\e^{-W(x)}\,\d x, \end{align*} where the last inequality follows from \eqref{ee-00} again.

Combining both inequalities above, we have
\begin{equation}\label{NNK}\begin{split}\mu_W(f^2)\le &\left(r+\frac{2}{\kappa_W(B_{l-k}^c)} \right)\int_{\R^n\times \R^n} \ff{|f^2(x)-f^2(y)|}{|x-y|^{n+\alpha/2}}\,\d y\,\e^{-W(x)}\,\d x\\
&+ c_1\left(r^{-2n/\aa}+r^{-n}\right)\left(\Hup_{|z|\le l+k}\e^{W(z)/2}\right) \mu_W(|f|)^2\\
&+\left(\Big(\frac{ c_2}{k}+ c_3 \Hup_{|z|\le l+k+1} \e^{2|\nabla W(z)|}\Big) r  + \frac{c_2}{k\,\kappa_W(B_{l-k}^c)}\right)\mu_W(f^2). \end{split}\end{equation}

(1) Taking $l=2k$, we have
\begin{align*}&\Big(\frac{ c_2}{k}+ c_3 \Hup_{|z|\le l+k+1} \e^{2|\nabla W(z)|}\Big) r  + \frac{c_2}{k\,\kappa_W(B_{l-k}^c)}\\
&= \Big(\frac{ c_2}{k}+ c_3 \Hup_{|z|\le 3k+1} \e^{2|\nabla W(z)|}\Big) r  + \frac{c_2}{k\,\kappa_W(B_{k}^c)}.\end{align*}
Since $\kappa_W(B_{l}^c)$ is increasing with respect to $l$,  under \eqref{e:ps11} we can choose $k\ge 1$ large enough and then take $r>0$ small enough such that
$$ \Big(\frac{ c_2}{k}+ c_3 \Hup_{|z|\le 3k+1} \e^{2|\nabla W(z)|}\Big) r  + \frac{c_2}{k\,\kappa_W(B_{k}^c)}\le \frac{1}{2}.$$
This along with \eqref{NNK} yields \eqref{e:the1}.

(2) Taking $k=1$ in \eqref{NNK} and using \eqref{e:ps12}, we know that \eqref{e:the2} holds with
\begin{align*}\beta(r)=\inf\bigg\{&2c_1\left(s^{-2n/\aa}+s^{-n}\right)\left(\Hup_{|z|\le l+1}\e^{W(z)/2}\right):s+\frac{2}{\kappa_W(B_{l-1}^c)} \le \frac{r}{2}\\
& \quad \text{ and } \Big( c_2+ c_3 \Hup_{|z|\le l+2} \e^{2|\nabla W(z)|}\Big) s  + \frac{c_2}{\kappa_W(B_{l-1}^c)}\le \ff 1 2  \bigg\}.\end{align*} Note that
\begin{align*}\beta(r)\le \inf\bigg\{&2c_1\left(s^{-2n/\aa}+s^{-n}\right)\left(\Hup_{|z|\le l+1}\e^{W(z)/2}\right):s+\frac{1}{\kappa_W(B_{l-1}^c)} \le c_4({r}\wedge1)\\
& \qquad\qquad\qquad\qquad\qquad\qquad \qquad\qquad\qquad  \text{ and } \Hup_{|z|\le l+2} \e^{2|\nabla W(z)|} s \le c_5  \bigg\}.\end{align*} Then, we prove \eqref{e:the2} with the desired $\beta$.
\end{proof}

\begin{lem}\label{L:2} For any  $\aa\in (0,2)$, there is a constant $c_0>0$ such that   $$\kappa_W(B_{l}^c)\ge c_0\inf_{|x|\ge l} \frac{\e^{W(x)}}{|x|^{n+\aa/2}},\ \ l\ge 1.$$
 \end{lem}
\begin{proof}According to the definition of $\kappa_W(B_{l}^c)$, we have
\begin{align*}\kappa_W(B_{l}^c)\ge &\frac{1}{2}\inf_{A\Subset B_l^c} \inf_{x\in A}\int_{A^c} \frac{\e^{W(x)-W(y)}+1}{|x-y|^{n+\aa/2}}\,\d y\ge \frac{1}{2}\inf_{|x|\ge l}\int_{\{|y|< l\}} \frac{\e^{W(x)-W(y)}+1}{|x-y|^{n+\aa/2}}\,\d y\\
\ge & c_1\inf_{|x|\ge l} \e^{W(x)}\int_{\{|y|\le 1\}} \frac{1}{|x-y|^{n+\aa/2}}\,\d y\ge c_2\inf_{|x|\ge l} \frac{\e^{W(x)}}{|x|^{n+\aa/2}}.\end{align*} This proves the desired assertion.\end{proof}

Theorem \ref{T:thm} is a direct consequence of Lemmas \ref{L:1} and \ref{L:2}, and so we omit the proof here.

The example below indicates that Theorem \ref{T:thm} is sharp in some sense.
\begin{exa}Let $W(x) = \ff 1 2 (n + \varepsilon) \log(1 + |x|^2) +c_{n,\vv}$ for $\vv>0$.
\begin{itemize}
\item[$(1)$]\eqref{e:the1} holds  if and only if $\vv\ge \aa/2$.
\item[$(2)$] \eqref{e:the2} holds if and only if $\vv>\aa/2$. Furthermore, when $\vv>\aa/2$, \eqref{e:the2} holds with
$$\bb(r)=c_1(1\wedge r)^{-\frac{2n}{\aa}-\frac{n+\vv}{2\vv-\aa}}.$$ \end{itemize}
\end{exa}
\begin{proof}The sufficiency for both conclusions is easily seen from Theorem \ref{T:thm}. To verify the necessary, we will make use of the  reference functions used in \cite[Corollary 1.1]{WW}. For any $l\ge 1$, define $g_l\in C^\infty(\R^n)$ such that $|\nn g_l|\le 2/l$ and
$$g_l(x)\beg{cases} =0, &\text{if}\ |x|\le l, \\
\in [0,1], &\text{if} \ |x|\in [l, 2l],\\
=1, &\text{if}\ |x|\ge 2l.\end{cases}$$
Then there exists a constant $c_0>0$ independent of $l$ such that for all $x\in \R^n$ and $l\ge 1$,
\beq\label{NN}\beg{split}  \int_{\R^n} \ff{|g_l^2(y)-g_l^2(x)|}{|x-y|^{n+\aa/2}}\, \d y &\le 2\int_{\R^n} \ff{|g_l(y)-g_l(x)|}{|x-y|^{n+\aa/2}}\, \d y\\
&\le \ff 4{l}\int_{\{|x-y|\le l\}} \ff 1 {|y-x|^{n+\aa/2-1}}\,\d y+2\int_{\{|x-y|\ge l\}} \ff 1 {|x-y|^{n+\aa/2}}\,\d y\\
&\le \ff {c_0} {l^{\aa/2}}.\end{split}\end{equation}
Obviously,
\beq\label{NNDD}\mu_W(g^2_l) \ge \ff{c_1}{l^\vv},\ \  \mu_W(g_l)^2\le \ff{c_2}{l^{2\vv}},\quad l\ge 1\end{equation} hold for some constants $c_1,c_2>0$. Note that, since $1-g_l\in C_c^\infty(\R^n)$, we can directly apply $g_l$ into \eqref{e:the1} and \eqref{e:the2}.

(1) Combining (\ref{NN}) with \eqref{NNDD}, we see that for any $c>0$,
$$\lim_{l\to\infty} \ff{ 1}{\mu_W(g^2_l)-c\mu_W(g_l)^2}\int \ff{|g_l^2(y)-g_l^2(x)|}{|x-y|^{n+\aa/2}}\, \d y\,\mu_W(\d x) \le \lim_{l\to\infty} \ff{c_0l^{-\alpha/2}}{c_1 l^{-\vv}- cc_2 l^{-2\vv}}=0$$
provided $\vv\in (0,\aa/2).$ Thus, the inequality \eqref{e:the1} does not hold.

 (2) We first prove that if $\vv\le\aa/2$, then for any $\bb: (0,\infty)\to (0,\infty)$ the inequality (\ref{e:the2}) does not hold. Indeed, if this inequality holds, then, by \eqref{NN} and \eqref{NNDD},
$$\ff{c_1}{l^\vv}\le \ff{c_0r}{l^{\aa/2}} + \ff{c_2\bb(r)}{l^{2\vv}},\quad r>0, l\ge 1$$ holds for some constants $c_0,c_1,c_2>0$.
Since $\vv \in (0, \aa/2]$, we obtain
$$c_1\le \lim_{l\to\infty} \ff{c_0r}{l^{(\aa/2)-\vv}} +\lim_{l\to\infty} \ff{c_2\bb(r)}{l^\vv} \le c_0r,\quad  r>0.$$ Letting $r\to 0$ we conclude that $c_1\le 0$, which is however impossible.
Furthermore, by   Theorem \ref{T:thm}(2), it is easy to prove \eqref{e:the2} with the desired rate function $\bb(r)$.\end{proof}

\paragraph{Acknowledgement.} Supported in
part by  NNSFC (11431014,11522106,11626245,11626250), the Fok Ying Tung
Education Foundation (151002),
National Science Foundation of
Fujian Province (2015J01003) and the Program for Probability and Statistics: Theory and Application (No.\ IRTL1704). The authors would like to thank Professor Takashi Kumagai for helpful comments.

\end{document}